\documentclass[a4paper]{extarticle}
\usepackage[utf8]{inputenc}
\usepackage[english]{babel}
\usepackage{amsmath}
\usepackage{amssymb}
\usepackage{amsfonts}
\usepackage{amsthm}

\usepackage{tikz}  
\usepackage{tikz-cd}  
\usetikzlibrary{
	decorations.markings
}
\usepackage{indentfirst}  

\newtheoremstyle{blockstyle} 
{} 
{0.5em} 
{} 
{\parindent} 
{\bfseries} 
{.} 
{.5em} 
{} %

\theoremstyle{blockstyle}
\newtheorem{theorem}{Theorem}
\newtheorem{lemma}{Lemma}
\newtheorem{definition}{Definition}
\newtheorem*{example}{Example}
\newtheorem{remark}{Remark}
\newtheorem{proposition}{Proposition}

\usepackage{enumitem}
\setlist{
label*={\arabic*.},
}
\setlist[1]{
wide=0pt, 
leftmargin=5pt, 
itemindent=\parindent, 
labelsep=5pt
}

\usepackage{titlesec}
\titlespacing*{\section}{\parindent}{10pt}{0pt}
\titlespacing*{\subsection}{\parindent}{10pt}{5pt}
\titlespacing*{\subsubsection}{\parindent}{10pt}{0pt}

\titleformat{\section}{\Large\bfseries}{\thesection.}{5pt}{}
\titleformat{\subsection}{\large\bfseries}{\thesubsection.}{5pt}{}
\titleformat{\subsubsection}{\normalsize\bfseries}{\thesubsubsection.}{5pt}{}

\renewenvironment{proof}{{\textit{Proof.}}}{\hfill $\square$}

\usepackage{tocloft}

\AtBeginDocument{}

\usepackage[symbol]{footmisc}

\usepackage[labelsep=period]{caption}

\usepackage[left=20mm,right=20mm,top=20mm,bottom=20mm,bindingoffset=0mm]{geometry}

\newcommand{\K}{\mathbb{K}}
\newcommand{\I}{\mathcal{I}}
\newcommand{\Z}{\mathbb{Z}}
\newcommand{\Sy}{\mathcal{S}}
\newcommand{\Zmap}{\mathcal{Z}}
\newcommand{\ksmall}{k}

\tikzset{line/.style={draw=black, thick}}
\tikzset{line_dashed/.style={draw=black, thick, dashed}}
\tikzset{orientation_line/.style={draw=gray, thin, ->}}
\tikzset{orientation_line_opp/.style={draw=gray, thin, <-}}

\begin{document}

\begin{center}
\textbf{{\large Symmetric Dijkgraaf -- Witten type invariants for 3-manifolds}}

Korablev Ph. G.\footnote{The research was supported by RSF (project No. 23-21-10014)}

\emph{korablev@csu.ru}
\end{center}

\begin{abstract}
In the paper we introduce the construction of invariants for 3-manifolds, based on the same key concepts as the classical Dijkgraaf -- Witten invariant. We introduce the notion of a special $G$-system and describe how each system induces the invariant not only for closed 3-manifolds, but also for manifolds with boundary. Finally, we show how to construct a very simple one-dimensional special $G$-system using group cohomologies.
\end{abstract}

\tableofcontents

\section{Introduction}

Dijkgraaf -- Witten invariants for 3-manifolds were introduced in \cite{DW}. Later it was reformulated in terms of triangulations of oriented 3-manifolds in \cite{W}. There are several approaches to generalising the construction of Dijkgraaf -- Witten invariants. For example, in \cite{K} these invariants are extended to the class of 3-manifolds with cusps. In \cite[Appendix H]{T} Dijkgraaf -- Witten invariants are extracted from a spherical fusions category.

In this paper we introduce so-called symmetric Dijkgraaf -- Witten type invariants for 3-manifolds. Instead of triangulations we use special and simple spines (see \cite{M}) of 3-manifolds. For each simple polyhedron we consider the set of colourings of this polyhedron by elements of the finite group $G$. If the polyhedron is a special spine of a closed manifold, then these colourings coincide with representations from the fundamental group of the manifold to the group $G$. In general, the set of colours is correctly defined for any simple polyhedron. As for the classical Dijkgraaf -- Witten invariants, we assign a value from a ring $\K$ to each colouring. The difference is that in the classical case the value of the invariant is a sum of the assigned values, but in our case the invariant is a multi-set of these values, because this multi-set does not depend on the choice of the simple spine of the manifold. That's why we call these invariants Dijkgraaf -- Witten type invariants.

The advantage of our construction is that it gives invariants not only for closed 3-manifolds, but also for manifolds with boundary and even for virtual manifolds (\cite{VM}). The approach we use is an extension of the approach for constructing configuration invariants (\cite{KC}).

The key concept of our construction is the $G$-system. Roughly speaking, the $G$-system is a collection of modules and tensors on these modules that satisfy some symmetric properties. If this system satisfies one condition, then it is called a special $G$-system and defines the invariant of simple polyhedrons modulo $T$-moves. If it additionally satisfies another condition, then it is called a strong special $G$-system and defines the invariant modulo $T,L$-moves. The most important step in constructing the invariant is to find a special (or strong special) $G$-system. We do not provide a universal method for constructing these systems. It will be interesting to find the way to extract these systems from categories of certain types.

In section 2 we introduce the notion of the $G$-system and prove that if this system is strong and special, then it defines the invariant of 3-manifolds. In Section 3 we consider in more detail the simplest 1-dimensional case. In this case, all modules in the $G$ system are one dimensional. We construct a chain (and corresponding cochain) complex and show that any element from the third cohomology group defines the special $G$ system. This construction is similar to the classical cohomology construction for groups, but slightly different. For example, the well-known cocycles for cyclic groups are not cocycles in our case. At the end of the section we show an example with a non-trivial cocycle for the group $\mathbb{Z}_4$ and calculate the value of the invariant for the lens space $L_{4, 1}$.

\section{Invariant $DW_{\mathcal{S}}$}

\subsection{Simple polyhedrons}

\begin{definition}
\label{Definition:SpecialSpine}
A two-dimensional polyhedron $P$ is called \emph{simple} if the link of every point $x\in P$ is homeomorphic to one of the following 1-dimensional polyhedra:
\begin{enumerate}
\item A circle (figure \ref{Figure:SpecialPolyhedron} on the left), in this case the point $x$ is called \emph{regular point};
\item A circle with one diameter (figure \ref{Figure:SpecialPolyhedron} in the centre), in this case the point $x$ is called \emph{triple point};
\item A circle with two diameters (figure \ref{Figure:SpecialPolyhedron} on the right), in this case the point $x$ is called \emph{true vertex}.
\end{enumerate}

If the union of regular points of a simple polyhedron is a disjoint union of open discs, and the union of triple points is a disjoint union of intervals, then this polyhedron is called \emph{special}.
\end{definition}

\begin{figure}[h]
\begin{center}
\ \hfill
\begin{tikzpicture}[scale=0.5, baseline={([yshift=-10.0ex]current bounding box.center)}]
	\draw[line] (0, 0) -- (2, 2) -- (5, 2) -- (3, 0) -- (0, 0);
	\draw[fill] (2.5, 1) node[below] {$x$} circle (0.075);
\end{tikzpicture}
\hfill
\begin{tikzpicture}[scale=0.5]
	\draw[line] (0, 0) -- (2, 2) -- (5, 2) -- (3, 0) -- (0, 0);
	\draw[line] (3.5, 2) -- (1.5, 0) -- (1.5, -2) -- (3.5, 0) -- (3.5, 0.5);
	\draw[line_dashed] (3.5, 0.5) -- (3.5, 2);
	\draw[fill] (2.5, 1) node[below right] {$x$} circle (0.075);
\end{tikzpicture}
\hfill
\begin{tikzpicture}[scale=0.5]
	\draw[line] (1, 1) -- (0, 0) -- (3, 0) -- (5, 2) -- (4, 2);
	\draw[line_dashed] (1, 1) -- (2, 2) -- (4, 2);
	\draw[line] (1, 1) -- (1, 3) -- (4, 3) -- (4, 1) -- (1, 1);
	\draw[line] (2.5, 1) -- (1.5, 0) -- (1.5, -2) -- (3.5, 0) -- (3.5, 0.5);
	\draw[line_dashed] (2.5, 1) -- (3.5, 2) -- (3.5, 0.5);
	\draw[fill] (2.5, 1) node[below right] {$x$} circle (0.075);
\end{tikzpicture}
\hfill\ \ 
\end{center}
\caption{\label{Figure:SpecialPolyhedron}Regular point (on the left), triple point (on the centre) and true vertex (on the right)}
\end{figure}
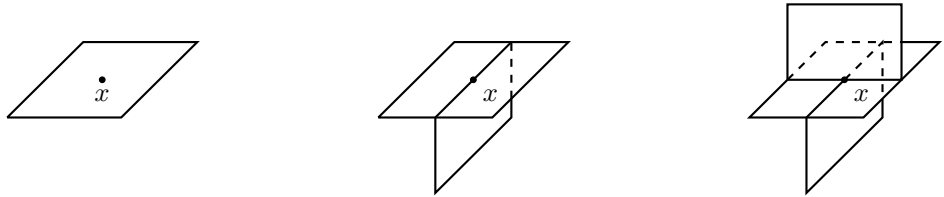

For each simple polyhedron $P$ denote $\mathcal{V}(P)$ the set of true vertices of $P$, $\mathcal{E}(P)$ the set of all triple lines (components of the union of all triple points), $\mathcal{C}(P)$ the set of all 2-components (components of the union of all regular points). If $P$ is special, then all elements of $\mathcal{C}(P)$ are discs and all components of $\mathcal{E}(P)$ are intervals. If $P$ is a simple, then the set $\mathcal{C}(P)$ can contain any surfaces.

\begin{definition}
\label{Definition:Spine}
Let $M$ be a 3-manifold. The polyhedron $P\subseteq M$ is called \emph{spine} of $M$ if the following holds:
\begin{enumerate}
\item In the case $\partial M = \emptyset$, the complement $M\setminus P$ is homeomorphic to an open 3-ball;
\item In the case $\partial M\neq \emptyset$, the complement $M\setminus P$ is homeomorphic to the direct product $\partial M\times [0; 1)$.
\end{enumerate}

The spine $P$ is called \emph{special spine} if the polyhedron $P$ is special, and it is called \emph{simple spine} if the polyhedron $P$ is simple.
\end{definition}

There are several types of local transformations of simple spines that do not change the corresponding manifold. We will need two transformations: $T$-move (figure \ref{Figure:TLMove} on the left) and $L$-move (figure \ref{Figure:TLMove} on the right).

\begin{figure}[h]
\begin{center}
\ \hfill
\begin{tikzpicture}[scale=0.4]
	\draw[line] (2, 7) -- (0, 7) -- (0, 0) -- (2, 0) -- (3.5, 2.5) -- (5, 0) -- (7, 0) -- (7, 7) -- (5, 7) -- (3.5, 4.5) -- cycle;
	\draw[line] (2, 7) arc (-180:0:1.5 and 0.75);
	\draw[line] (5, 7) arc (0:180:1.5 and 0.75);
	
	\draw[line] (2, 0) arc (-180:0:1.5 and 0.75);
	\draw[line_dashed] (5, 0) arc (0:180:1.5 and 0.75);
	
	\draw[line] (3.5, 2.5) -- (3.5, 4.5);
	
	\draw[line] (4, 7.7) -- (6, 9) -- (6, 7);
	\draw[line_dashed] (3.5, 4.5) -- (4, 7.7);
	\draw[line_dashed] (6, 7) -- (6, 2) -- (4, 0.7) -- (3.5, 2.5);
	
	\draw[fill] (3.5, 2.5) circle (0.075);
	\draw[fill] (3.5, 4.5) circle (0.075);
\end{tikzpicture}
\begin{tikzpicture}[scale=0.5, baseline={([yshift=-11.0ex]current bounding box.center)}]
	\draw[-latex] (0, 0.5) -- (2, 0.5) node[above, midway] {$T$};
	\draw[-latex] (2, -0.5) -- (0, -0.5) node[below, midway] {$T^{-1}$};
\end{tikzpicture}
\begin{tikzpicture}[scale=0.4]
	\draw[line] (2, 0) -- (2, 7) -- (0, 7) -- (0, 0) -- cycle;
	\draw[line] (7, 0) -- (7, 7) -- (5, 7) -- (5, 0) -- cycle;
	
	\draw[line] (2, 0) arc (-180:0:1.5 and 0.75);
	\draw[line] (2, 3.5) arc (-180:0:1.5 and 0.75);
	\draw[line] (2, 7) arc (-180:0:1.5 and 0.75);
	
	\draw[line_dashed] (5, 0) arc (0:180:1.5 and 0.75);
	\draw[line_dashed] (5, 3.5) arc (0:180:1.5 and 0.75);
	\draw[line] (5, 7) arc (0:180:1.5 and 0.75);
	
	\draw[line] (4, 7.7) -- (6, 9) -- (6, 7);
	\draw[line_dashed] (6, 7) -- (6, 2) -- (4, 0.7) -- (4, 7.7);
	
	\draw[fill] (2, 3.5) circle (0.075);
	\draw[fill] (5, 3.5) circle (0.075);
	\draw[fill] (4, 4.2) circle (0.075);
\end{tikzpicture}
\hfill
\begin{tikzpicture}[scale=0.4]
	\draw[line] (2, 7) -- (0, 7) -- (0, 0) -- (2, 0);
	\draw[line] (5, 7) -- (7, 7) -- (7, 0) -- (5, 0);
	
	\draw[line] (2, 0) arc (-180:0:1.5 and 0.75);
	\draw[line] (2, 7) arc (-180:0:1.5 and 0.75);
	
	\draw[line_dashed] (5, 0) arc (0:180:1.5 and 0.75);
	\draw[line] (5, 7) arc (0:180:1.5 and 0.75);
	
	\draw[line] (5, 0) arc (0:180:1.5);
	\draw[line] (2, 7) arc (-180:0:1.5);
\end{tikzpicture}
\begin{tikzpicture}[scale=0.5, baseline={([yshift=-11.0ex]current bounding box.center)}]
	\draw[-latex] (0, 0.5) -- (2, 0.5) node[above, midway] {$L$};
	\draw[-latex] (2, -0.5) -- (0, -0.5) node[below, midway] {$L^{-1}$};
\end{tikzpicture}
\begin{tikzpicture}[scale=0.4]
	\draw[line] (2, 0) -- (2, 7) -- (0, 7) -- (0, 0) -- cycle;
	\draw[line] (7, 0) -- (7, 7) -- (5, 7) -- (5, 0) -- cycle;
	
	\draw[line] (2, 0) arc (-180:0:1.5 and 0.75);
	\draw[line] (2, 3.5) arc (-180:0:1.5 and 0.75);
	\draw[line] (2, 7) arc (-180:0:1.5 and 0.75);
	
	\draw[line_dashed] (5, 0) arc (0:180:1.5 and 0.75);
	\draw[line_dashed] (5, 3.5) arc (0:180:1.5 and 0.75);
	\draw[line] (5, 7) arc (0:180:1.5 and 0.75);
	
	\draw[fill] (2, 3.5) circle (0.075);
	\draw[fill] (5, 3.5) circle (0.075);
\end{tikzpicture}
\hfill\ \ 
\end{center}
\caption{\label{Figure:TLMove}$T^{\pm 1}$-move (on the left) and $L^{\pm 1}$-move (on the right)}
\end{figure}
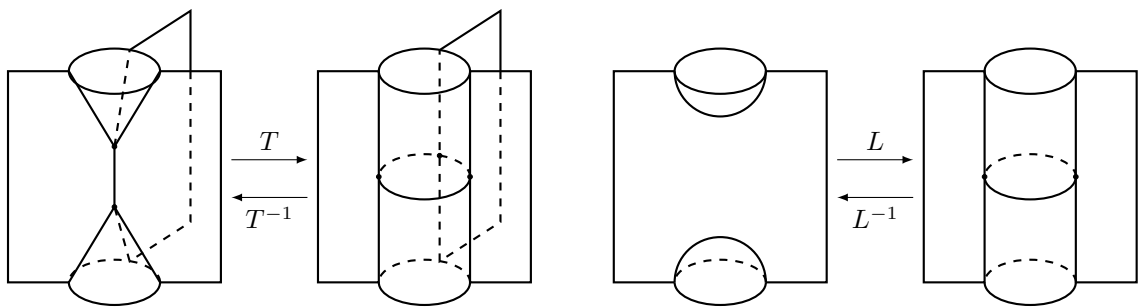

It is known that any two simple spines $P_1, P_2$ of the manifold $M$ can be obtained from each other by a finite sequence of $T$-moves, $L$-moves and  their inverse (\cite[Theorem 1.2.27]{M}). Furthermore, if both $P_1$ and $P_2$ are special and contain at least two true vertices, then it's sufficient to use only $T$-moves (\cite[Theorem 1.2.5]{M}) to transform $P_1$ into $P_2$.

\subsection{Colourations}

By an oriented simple polyhedron we mean a polyhedron with fixed orientations of all triple lines and 2-components.

\begin{definition}
\label{Definition:Colouring}
Let $M$ be an oriented 3-manifold, $P\subseteq M$ be a simple oriented spine, and $G$ be a finite group. The map $\xi\colon \mathcal{C}(P)\to G$ is called \emph{colouring} of $P$ if for every triple line $e\in \mathcal{E}(P)$: $$g_{i_1}^{\varepsilon_1}g_{i_2}^{\varepsilon_2}g_{i_3}^{\varepsilon_3} = 1,$$ where $g_{i_1}, g_{i_2}, g_{i_3}\in G$ are colours of 2-components incident to $e$, and $\varepsilon_j = 1$ if the orientation of the 2-component with the colour $g_{i_j}$ coincides with the orientation of $e$ and $\varepsilon_j = -1$ in the opposite case (figure \ref{Figure:Colouring}). The sequence of group elements $g_{i_1}, g_{i_2}, g_{i_3}$ is written with respect to the orientation of $e$ and the orientation of the whole manifold $M$.
\end{definition}

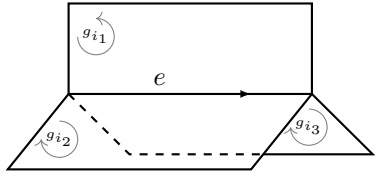
\begin{figure}[h]
\begin{center}
\begin{tikzpicture}[scale=0.8]
	\draw[line] (0, 0) -- (0, 1.5) -- (4, 1.5) -- (4, 0) -- cycle;
	\draw[line] (0, 0) -- (-1, -1.25) -- (3, -1.25) -- (4, 0);
	\draw[line_dashed] (0, 0) -- (1, -1) -- (3.2, -1);
	\draw[line] (3.2, -1) -- (5, -1) -- (4, 0);
	
	\draw[orientation_line] (0.15, 0.95) arc (-180:90:0.3);
	\draw (0.45, 0.95) node {{\tiny $g_{i_1}$}};
	
	\draw[orientation_line_opp] (-0.45, -0.75) arc (-180:90:0.3);
	\draw (-0.15, -0.75) node {{\tiny $g_{i_2}$}};
	
	\draw[orientation_line_opp] (3.65, -0.55) arc (-180:90:0.3);
	\draw (3.95, -0.55) node {{\tiny $g_{i_3}$}};
	
	\draw[-latex] (0, 0) -- (3, 0) node[above, midway] {$e$};
\end{tikzpicture}
\end{center}
\caption{\label{Figure:Colouring}Colours of 2-components in the neighbourhood of the edge $e$}
\end{figure}

Let $Col_{G}(P)$ be the set of all colourings of the simple polyhedron $P$ by elements of the group $G$. It's clear that the set $Col_{G}(P)$ is finite.

\begin{theorem}
\label{Theorem:ColouringCount}
Let $M$ be an oriented 3-manifold, $G$ be a finite group. Then for any two simple spines $P_1, P_2$ of the manifold $M$: $$|Col_{G}(P_1)| = |Col_G(P_2)|.$$
\end{theorem}
\begin{remark}
If $M$ is closed, then the statement of the theorem is obvious. Indeed, in this case for $i = 1, 2$ the set $Col_{G}(P_i)$ is just a set of homomorphisms from $\pi_1(P_i)$ to $G$. But in the more general case, when the manifold $M$ is not closed, this simple argument does not work.
\end{remark}
\begin{proof}
We are going to start by proving that the cardinality of the set $Col_{G}(P)$ does not depend on the orientation of $P$.

Let $P_2$ be obtained from $P_1$ by changing the orientation of a 2-component $c\in \mathcal{C}(P_1)$. Then for each colouring $\xi\colon\mathcal{C}(P_1)\to G$ there exists a corresponding colouring $\xi'\colon\mathcal{C}(P_2)\to G$, defined by the following rule: $\xi(c) = \xi'(c)^{-1}$, and $\xi(f) = \xi'(f)$ for all $f\neq c$, $f\in \mathcal{C}(P_1), \mathcal{C}(P_2)$.

Next consider the case where $P_2$ is obtained from $P_1$ by changing the orientation of a triple line $e\in\mathcal{E}(P_1)$. If $\xi\colon \mathcal{C}(P_1)\to G$ is a colouring, then the same map defines the colouring of the polyhedron $P_2$. It follows from the fact that $g_{i_1}^{\varepsilon_1} g_{i_2}^{\varepsilon_2} g_{i_3}^{\varepsilon_3} = 1$ is equivalent to $g_{i_3}^{- \varepsilon_3} g_{i_2}^{-\varepsilon_2} g_{i_1}^{-\varepsilon_1} = 1$ in $G$.

It's known that any two simple spines of the manifold $M$ can be obtained from each other by a finite sequence of $T$-moves, $L$-moves and their inverse (\cite[Theorem 1.2.27]{M}).

Let $P_2$ be obtained from $P_1$ by a $T$-move. Fix the orientations of $P_1$ and $P_2$ as shown in the figure \ref{Figure:TMoveColourings}. Let $\xi\colon \mathcal{C}(P_1)\to G$ be a colouring as shown in the figure \ref{Figure:TMoveColourings} on the left. Then in particular
\begin{center}
\begin{tabular}{lll}
$x_1 y_3^{-1} y_2 = 1$, & $x_2 y_1^{-1} y_3 = 1$, & $x_3 y_2^{-1} y_1 = 1$, \\
$x_1 z_3 z_2^{-1} = 1$, & $x_2 z_1 z_3^{-1} = 1$, & $x_3 z_2 z_1^{-1} = 1$, \\
 & $x_1 x_2 x_3 = 1$. &
\end{tabular}
\end{center}

It follows that $y_1^{-1}y_3 = x_2^{-1}$, $y_1^{-1} y_2 = x_3$, $x_2 = z_3 z_1^{-1}$ and $x_3 = x_1 x_1^{-1}$.

Define the corresponding colouring $\xi'\colon\mathcal{C}(P_2)\to G$ as follows. All 2-components are orientated and coloured by the same colours, and the new 2-component is orientated as shown in the figure \ref{Figure:TMoveColourings} on the right and coloured by the element $w\in G$. We should prove that there is a unique colour $w$ such that the following equations hold
\begin{center}
\begin{tabular}{lll}
$x_1 y_3^{-1} y_2 = 1$, & $x_2 y_1^{-1} y_3 = 1$, & $x_3 y_2^{-1} y_1 = 1$, \\
$x_1 z_3 z_2^{-1} = 1$, & $x_2 z_1 z_3^{-1} = 1$, & $x_3 z_2 z_1^{-1} = 1$, \\
$w y_3 z_3 = 1$, & $w y_1 z_1 = 1$, & $w y_2 z_2 = 1$.
\end{tabular}
\end{center}

The first six equations are the same as for $\xi$, so it holds. Define $w = z_1^{-1} y_1^{-1}$. Note that $z_1^{-1}y_1^{-1}y_3 z_3 = z_1^{-1} x_2^{-1} z_3 = 1$ and $z_1^{-1} y_1^{-1} y_2 z_2 = z_1^{-1} x_3 z_2 = 1$. So the colouring $\xi'$ is correctly defined.

To prove that the colouring $\xi$ is also uniquely defined from $\xi'$, note that from $x_1 = y_2^{-1} y_3$, $x_2 = y_3^{-1} y_1$ and $x_3 = y_1^{-1} y_2$ follows that $x_1 x_2 x_3 = 1$.

\begin{figure}[h]
\begin{center}
\ \hfill
\begin{tikzpicture}[scale=0.4]

	\draw[line] (2, 7) -- (0, 7) -- (0, 0) -- (2, 0) -- (4.5, 2.5) -- (7, 0) -- (9, 0) -- (9, 7) -- (7, 7) -- (4.5, 4.5) -- cycle;
	\draw[line] (2, 7) arc (-180:0:2.5 and 0.75);
	\draw[line] (7, 7) arc (0:180:2.5 and 0.75);
	
	\draw[line] (2, 0) arc (-180:0:2.5 and 0.75);
	\draw[line_dashed] (7, 0) arc (0:180:2.5 and 0.75);
	
	\draw[line] (4.5, 2.5) -- (4.5, 4.5);
	
	\draw[line] (5, 7.73) -- (7.5, 9) -- (7.5, 7);
	\draw[line_dashed] (7.5, 7) -- (7.5, 2) -- (5, 0.73);
	
	\draw[fill] (4.5, 2.5) circle (0.075);
	\draw[fill] (4.5, 4.5) circle (0.075);
	
 	\draw[-latex, thick] (4.5, 3) -- (4.5, 4);
 	\draw[-latex, thick] (4, 5) -- (3, 6);
 	\draw[-latex, thick] (5, 5) -- (6, 6);
 	\draw[-latex, thick] (3, 1) -- (4, 2);
 	\draw[-latex, thick] (6, 1) -- (5, 2);
 	\draw[-latex, thick, dashed] (4.5, 4.5) -- (5, 7.73);
 	\draw[-latex, thick, dashed] (5, 0.73) -- (4.5, 2.5);
	
	\draw[orientation_line] (0.5, 6) arc (-180:90:0.5);
	\draw (1.0, 6.0) node {{\tiny $x_1$}};
	
	\draw[orientation_line_opp] (7.5, 1) arc (-180:90:0.5);
	\draw (8.0, 1.0) node {{\tiny $x_2$}};
	
	\draw[orientation_line_opp] (6.3, 8) arc (-180:90:0.5);
	\draw (6.8, 8.0) node {{\tiny $x_3$}};
	
	\draw[orientation_line_opp] (5.2, 7.0) arc (-180:90:0.5);
	\draw (5.7, 7.0) node {{\tiny $y_1$}};
	
	\draw[orientation_line_opp] (3.0, 7.0) arc (-180:90:0.5);
	\draw (3.5, 7.0) node {{\tiny $y_2$}};
	
	\draw[orientation_line] (3.9, 5.5) arc (-180:90:0.5);
	\draw (4.4, 5.5) node {{\tiny $y_3$}};
	
	\draw[orientation_line] (3.8, 1.5) arc (-180:90:0.4);
	\draw (4.2, 1.5) node {{\tiny $z_2$}};
	
	\draw[orientation_line_opp] (4.0, -0.15) arc (-180:90:0.5);
	\draw (4.5, -0.15) node {{\tiny $z_3$}};
	
	\draw[orientation_line] (5.0, 1.15) arc (-180:90:0.3);
	\draw (5.3, 1.15) node {{\tiny $z_1$}};
\end{tikzpicture}
\hfill
\begin{tikzpicture}[scale=0.4]
	
	\draw[line] (0, 0) -- (2, 0) -- (2, 7) -- (0, 7) -- cycle;
	\draw[line] (7, 0) -- (9, 0) -- (9, 7) -- (7, 7) -- cycle;
	\draw[line] (5, 7.73) -- (7.5, 9) -- (7.5, 7);
	\draw[line_dashed] (7.5, 7) -- (7.5, 2) -- (5, 0.73);
	\draw[-latex, thick, dashed] (5, 0.73) -- (5, 4.25);
	\draw[-latex, thick, dashed] (5, 4.25) -- (5, 7.73);
	
	\draw[line] (2, 7) arc (-180:0:2.5 and 0.75);
	\draw[line] (7, 7) arc (0:180:2.5 and 0.75);
	
	\draw[line] (2, 0) arc (-180:0:2.5 and 0.75);
	\draw[line_dashed] (7, 0) arc (0:180:2.5 and 0.75);
	\draw[line] (2, 3.5) arc (-180:0:2.5 and 0.75);
	\draw[line_dashed, -latex] (7, 3.5) arc (0:78:2.5 and 0.75);
	\draw[line_dashed, -latex] (5, 4.25) arc (90:125:3.0 and 0.55);
	\draw[line_dashed] (2, 3.5) arc (180:115:2.5 and 0.75);
	
	\draw[fill] (2, 3.5) circle (0.075);
	\draw[fill] (7, 3.5) circle (0.075);
	\draw[fill] (5, 4.25) circle (0.075);
	
	\draw[-latex, thick] (2, 1) -- (2, 3);
	\draw[-latex, thick] (2, 4) -- (2, 6);
	\draw[-latex, thick] (7, 1) -- (7, 3);
	\draw[-latex, thick] (7, 4) -- (7, 6);
	\draw[-latex, thick] (4.5, 2.75) arc (-90:-45:2.5 and 0.75);
	
	\draw[orientation_line] (0.5, 6) arc (-180:90:0.5);
	\draw (1.0, 6.0) node {{\tiny $x_1$}};
	
	\draw[orientation_line_opp] (7.5, 1) arc (-180:90:0.5);
	\draw (8.0, 1.0) node {{\tiny $x_2$}};
	
	\draw[orientation_line_opp] (6.3, 8) arc (-180:90:0.5);
	\draw (6.8, 8.0) node {{\tiny $x_3$}};
	
	\draw[orientation_line_opp] (5.2, 7.0) arc (-180:90:0.5);
	\draw (5.7, 7.0) node {{\tiny $y_1$}};
	
	\draw[orientation_line_opp] (3.0, 7.0) arc (-180:90:0.5);
	\draw (3.5, 7.0) node {{\tiny $y_2$}};
	
	\draw[orientation_line] (3.9, 5.5) arc (-180:90:0.5);
	\draw (4.4, 5.5) node {{\tiny $y_3$}};
	
	\draw[orientation_line] (2.5, 1.5) arc (-180:90:0.5);
	\draw (3, 1.5) node {{\tiny $z_2$}};
	
	\draw[orientation_line_opp] (4.0, -0.15) arc (-180:90:0.5);
	\draw (4.5, -0.15) node {{\tiny $z_3$}};
	
	\draw[orientation_line] (5.5, 1.7) arc (-180:90:0.5);
	\draw (6, 1.7) node {{\tiny $z_1$}};
	
	\draw[orientation_line] (3.7, 3.5) arc (-180:90:0.5);
	\draw (4.2, 3.5) node {{\tiny $w$}};
\end{tikzpicture}
\hfill\ \ 
\end{center}
\caption{\label{Figure:TMoveColourings}Coloured oriented polyhedrons before (on the left) and after (on the right) $T$-move}
\end{figure}
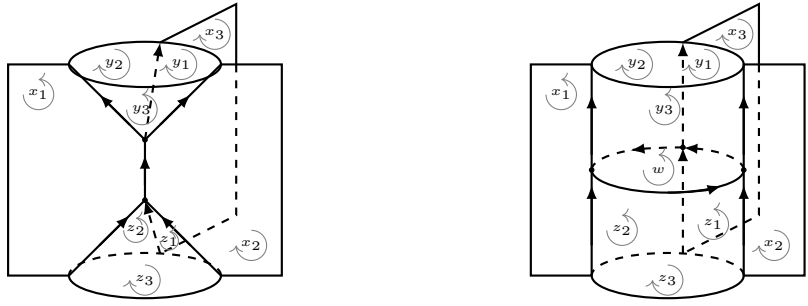

Let $P_2$ be obtained from $P_1$ by a $L$-move, and let $\xi\colon \mathcal{C}(P_1)\to G$ be the colouring of $P_1$ shown in the figure \ref{Figure:LMoveColouring}. It follows that $ab_1 c_1 = 1$ and $a b_2 c_2 = 1$. Define the colouring $\xi'\colon\mathcal{C}(P_2)\to G$ as shown in the figure \ref{Figure:LMoveColouring}. We should prove that $a' = a$ and there is a unique colour $w\in G$ such that
\begin{center}
\begin{tabular}{ll}
$ab_1 c_1 = 1$ & $a b_2 c_2 = 1$ \\
$c_2 c_1^{-1} w = 1$ & $b_2^{-1} b_1 w = 1$ \\
$a' b_2 c_ 2 = 1$ & $a' b_1 c_1 = 1$
\end{tabular}
\end{center}

Note that $a' = (b_2 c_2)^{-1} = (a^{-1})^{-1} = a$. Choose $w = c_1 c_2^{-1}$. Then $b_2^{-1} b_1 c_1 c_2^{-1} = b_2^{-1} a^{-1} c_2^{-1} = 1$.

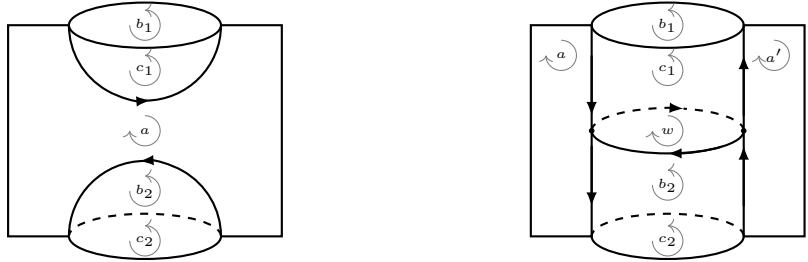
\begin{figure}[h]
\begin{center}
\ \hfill
\begin{tikzpicture}[scale=0.4]

	\draw[line] (2, 0) -- (0, 0) -- (0, 7) -- (2, 7);
	\draw[line] (7, 0) -- (9, 0) -- (9, 7) -- (7, 7);
	\draw[line] (2, 7) arc (-180:0:2.5 and 0.75);
	\draw[line] (7, 7) arc (0:180:2.5 and 0.75);
	
	\draw[line] (2, 0) arc (-180:0:2.5 and 0.75);
	\draw[line_dashed] (7, 0) arc (0:180:2.5 and 0.75);
	
	\draw[line] (2, 0) arc (180:90: 2.5);
	\draw[line, -latex] (7, 0) arc (0:94: 2.5);
	\draw[line, -latex] (2, 7) arc (-180:-86: 2.5);
	\draw[line] (7, 7) arc (0:-90: 2.5);
	
	\draw[orientation_line_opp] (4.0, 3.5) arc (-180:90:0.5);
	\draw (4.5, 3.5) node {{\tiny $a$}};
	
	\draw[orientation_line] (4.0, 7.0) arc (-180:90:0.5);
	\draw (4.5, 7.0) node {{\tiny $b_1$}};
	
	\draw[orientation_line] (4.0, 5.5) arc (-180:90:0.5);
	\draw (4.5, 5.5) node {{\tiny $c_1$}};
	
	\draw[orientation_line] (4.0, 1.5) arc (-180:90:0.5);
	\draw (4.5, 1.5) node {{\tiny $b_2$}};
	
	\draw[orientation_line] (4.0, -0.15) arc (-180:90:0.5);
	\draw (4.5, -0.15) node {{\tiny $c_2$}};
\end{tikzpicture}
\hfill
\begin{tikzpicture}[scale=0.4]
	
	\draw[line] (0, 0) -- (2, 0) -- (2, 7) -- (0, 7) -- cycle;
	\draw[line] (7, 0) -- (9, 0) -- (9, 7) -- (7, 7) -- cycle;
	
	\draw[line] (2, 7) arc (-180:0:2.5 and 0.75);
	\draw[line] (7, 7) arc (0:180:2.5 and 0.75);
	
	\draw[line] (2, 0) arc (-180:0:2.5 and 0.75);
	\draw[line_dashed] (7, 0) arc (0:180:2.5 and 0.75);
	\draw[line] (2, 3.5) arc (-180:0:2.5 and 0.75);
	\draw[line_dashed] (7, 3.5) arc (0:78:2.5 and 0.75);
	\draw[line_dashed, -latex] (2, 3.5) arc (180:78:2.5 and 0.75);
	
	\draw[fill] (2, 3.5) circle (0.075);
	\draw[fill] (7, 3.5) circle (0.075);
	
	\draw[latex-, thick] (2, 1) -- (2, 3);
	\draw[latex-, thick] (2, 4) -- (2, 6);
	\draw[-latex, thick] (7, 1) -- (7, 3);
	\draw[-latex, thick] (7, 4) -- (7, 6);
	\draw[latex-, thick] (4.5, 2.75) arc (-90:-45:2.5 and 0.75);
	
	\draw[orientation_line_opp] (0.5, 6) arc (-180:90:0.5);
	\draw (1.0, 6.0) node {{\tiny $a$}};
	
	\draw[orientation_line_opp] (7.5, 6) arc (-180:90:0.5);
	\draw (8.0, 6.0) node {{\tiny $a'$}};
	
	\draw[orientation_line] (4, 7.0) arc (-180:90:0.5);
	\draw (4.5, 7.0) node {{\tiny $b_1$}};
	
	\draw[orientation_line] (4.0, 5.5) arc (-180:90:0.5);
	\draw (4.5, 5.5) node {{\tiny $c_1$}};
		
	\draw[orientation_line] (4.0, -0.15) arc (-180:90:0.5);
	\draw (4.5, -0.15) node {{\tiny $c_2$}};
	
	\draw[orientation_line] (4.0, 1.7) arc (-180:90:0.5);
	\draw (4.5, 1.7) node {{\tiny $b_2$}};
	
	\draw[orientation_line_opp] (4.0, 3.5) arc (-180:90:0.5);
	\draw (4.5, 3.5) node {{\tiny $w$}};
\end{tikzpicture}
\hfill\ \ 
\end{center}
\caption{\label{Figure:LMoveColouring}Coloured oriented polyhedron before (on the left) and after (on the right) $L$-move}
\end{figure}
\end{proof}

\begin{remark}
\label{Remark:ProofColoring}
In fact, in the proof of the theorem \ref{Theorem:ColouringCount} we show a more general statement: there is a natural bijection between $Col_{G}(P_1)$ and $Col_{G}(P_2)$ for any two simple spines $P_1, P_2$ of the manifold $M$.
\end{remark}

\subsection{$G$-systems}

Let $G$ be a finite group and let $\K$ be an associative ring with unit. For each triple $(a, b, c)\in G^{3}$ such that $abc = 1$, consider the finite dimensional module $V_{a, b, c}$ over the ring $\K$. We will assume that $V_{a, b, c} = V_{b, c, a} = V_{c, a, b}$, all these three notations are just different names for the same module. Furthermore, let $\I_{a, b, c}\colon V_{a, b, c}\to V^{*}_{c^{-1}, b^{-1}, a^{-1}}$ be a fixed isomorphism between the module $V_{a, b, c}$ and the dual module $V_{c^{-1}, b^{-1}, a^{-1}}^{*}$.

\begin{definition}
\label{Definition:ModuleSystem}
The pair $(\{V_{a, b, c}\}, \{\I_{a, b, c}\})$ is called \emph{module system} for the group $G$.
\end{definition}

Let $G$ be a finite group, $\K$ a ring, and $(\{V_{a, b, c}\}, \{\I_{a, b, c}\})$ a module system for $G$. Consider the family of tensors $$[a, b, c]\colon V_{b, c, (bc)^{-1}}\otimes V^{*}_{ab, c, (abc)^{-1}}\otimes V_{a, bc, (abc)^{-1}}\otimes V^{*}_{a, b, (ab)^{-1}}\to \K,$$ defined for each triple $(a, b, c)\in G^3$, and the family of tensors $$\overline{[a, b, c]}\colon V^{*}_{b, c, (bc)^{-1}}\otimes V_{ab, c, (abc)^{-1}}\otimes V^{*}_{a, bc, (abc)^{-1}}\otimes V_{a, b, (ab)^{-1}}\to \K,$$ also defined for each triple $(a, b, c)\in G^{3}$.

Describe the geometric meaning of these tensors $[a, b, c]$ and $\overline{[a, b, c]}$. Let $T^{+} = [0, 1, 2, 3]$ be a tetrahedron with ordered vertices such that this ordering induces a positive orientation of the tetrahedron (figure \ref{Figure:Tetrahedrons} on the left). Orient all edges of $T^{+}$ from $i$ to $j$, where $i < j$. Colour the edges of $T^{+}$ with elements of the group $G$ as shown in the figure \ref{Figure:Tetrahedrons} on the left. The edges of the chain $0 - 1 - 2 - 3$ are coloured by elements $a, b, c\in G$, the colours of all other edges are uniquely defined by the following rule: for the face with vertices $i < j < k$, the colour of the edge $[i, k]$ is equal to the product of the colours of the edges $[i, j]$ and $[j, k]$.

\begin{figure}[h]
\begin{center}
\ \hfill
\begin{tikzpicture}[scale=0.75]
\begin{scope}[line, decoration={
    markings,
    mark=at position 0.5 with {\arrow{latex}}}
    ] 
    \draw[postaction={decorate}] (2, 3) -- (0, 0) node[left, midway] {$a$};
    \draw[postaction={decorate}] (0, 0) -- (2, -1) node[below left, midway] {$b$};
    \draw[postaction={decorate}] (2, -1) -- (4, 1) node[below right, midway] {$c$};
    \draw[postaction={decorate}] (2, 3) -- (4, 1) node[above right, midway] {$abc$};
    \draw[postaction={decorate}] (2, 3) -- (2, -1) node[right, pos=0.4] {$ab$};
    \draw[postaction={decorate}, line_dashed] (0, 0) -- (4, 1) node[above left, pos=0.4] {$bc$};
\end{scope}

	\draw[fill] (2, 3) circle (0.04) node[above] {0};
	\draw[fill] (0, 0) circle (0.04) node[left] {1};
	\draw[fill] (2, -1) circle (0.04) node[below] {2};
	\draw[fill] (4, 1) circle (0.04) node[right] {3};
\end{tikzpicture}
\hfill
\begin{tikzpicture}[scale=0.75]
\begin{scope}[line, decoration={
    markings,
    mark=at position 0.5 with {\arrow{latex}}}
    ] 
    \draw[postaction={decorate}] (2, 3) -- (0, 0) node[left, midway] {$a$};
    \draw[postaction={decorate}] (0, 0) -- (2, -1) node[below left, midway] {$bc$};
    \draw[postaction={decorate}] (4, 1) -- (2, -1) node[below right, midway] {$c$};
    \draw[postaction={decorate}] (2, 3) -- (4, 1) node[above right, midway] {$ab$};
    \draw[postaction={decorate}] (2, 3) -- (2, -1) node[right, pos=0.4] {$abc$};
    \draw[postaction={decorate}, line_dashed] (0, 0) -- (4, 1) node[above left, pos=0.4] {$b$};
\end{scope}

	\draw[fill] (2, 3) circle (0.04) node[above] {0};
	\draw[fill] (0, 0) circle (0.04) node[left] {1};
	\draw[fill] (2, -1) circle (0.04) node[below] {3};
	\draw[fill] (4, 1) circle (0.04) node[right] {2};
\end{tikzpicture}
\hfill\ \ 
\end{center}
\caption{\label{Figure:Tetrahedrons}Coloured tetrahedrons $T^{+}$ (on the left) and $T^{-}$ (on the right)}
\end{figure}
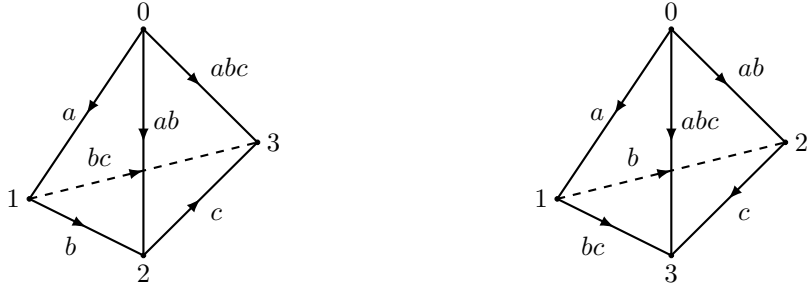

The tetrahedron $T^{+}$ has four faces $\partial_{0}T^{0} = [1, 2, 3]$, $\partial_1 T^+ = [0, 2, 3]$, $\partial_2 T^+ = [0, 1, 3]$ and $\partial_3 T^+ = [0, 1, 2]$. If the vertex order of the face $[i, j, k]$, where $i < j < k$, induces the inner normal to the face, then we assign to this face the module $V_{x, y, (xy)^{-1}}$, where $x$ is a colour of the edge $[i, j]$ and $y$ is a colour of the edge $[j ,k]$. If the order of the vertices induces an outer normal to the face $[i, j, k]$, then we assign the dual module to this face. So to the face $\partial_0 T^+$ is assigned the module $V_{b, c, (bc)^{-1}}$, to the face $\partial_1 T^+$ is assigned the module $V^{*}_{ab, c, (abc)^{-1}}$, to the face $\partial_2 T^+$ is assigned the module $V_{a, bc, (abc)^{-1}}$ and to the face $\partial_3 T^{+}$ is assigned the module $V^{*}_{a, b, (ab)^{-1}}$.

The tensor $[a, b, c]$ is a map from the tensor product of modules assigned to the faces $\partial_0 T^{+}$, $\partial_1 T^{+}$, $\partial_2 T^{+}$ and $\partial_3 T^{+}$ to the ring $\K$.

Tensors $\overline{[a, b, c]}$ defined by the similar approach, but instead of positive oriented tetrahedron we use the negative oriented tetrahedron $T^{-}$. The colouring of the edges of $T^{-}$ is shown in the figure \ref{Figure:Tetrahedrons} on the right. To the face $\partial_0 T^{-}$ we assign the module $V^{*}_{b, c, (bc)^{-1}}$, to the face $\partial_1 T^{-}$ --- the module $V_{ab, c, (abc)^{-1}}$, to the face $\partial_2 T^{-}$ --- the module $V^{*}_{a, bc, (abc)^{-1}}$ and to the face $\partial_3 T^{-}$ --- the module $V_{a, b, (ab)^{-1}}$.

\begin{definition}
\label{Definition:GSystem}
The quartet $(\{V_{a, b, c}\}, \{\I_{a, b, c}\}, \{[a, b, c]\}, \{\overline{[a, b ,c]}\})$ is called a \emph{$G$-system} if families of tensors $\{[a, b, c]\}$ and $\{\overline{[a, b, c]}\}$ satisfy the following conditions for all $a, b, c\in G$:
\begin{enumerate}
\item $\overline{[a^{-1}, ab, c]}(x_0\otimes x_1\otimes x_2\otimes x_3) = [a, b, c](x_1\otimes x_0\otimes \I^{-1}_{a, bc, (abc)^{-1}}(x_2)\otimes \I_{a^{-1}, ab, b^{-1}}(x_3))$ for any $x_0\in V^{*}_{ab, c, (abc)^{-1}}$, $x_1\in V_{b, c, (bc)^{-1}}$, $x_2\in V^{*}_{a^{-1}, abc, (bc)^{-1}}$ and $x_3\in V_{a^{-1}, ab, b^{-1}}$;
\item $\overline{[ab, b^{-1}, bc]}(x_0\otimes x_1\otimes x_2\otimes x_3) = [a, b, c](\I^{-1}_{b, c, (bc)^{-1}}(x_0)\otimes x_2\otimes x_1\otimes \I_{ab, b^{-1}, a^{-1}}(x_3))$ for any $x_0\in V^{*}_{b^{-1}, bc, c^{-1}}$, $x_1\in V_{a, bc, (abc)^{-1}}$, $x_2\in V^{*}_{ab, c, (abc)^{-1}}$ and $x_3\in V_{ab, b^{-1}, a^{-1}}$;
\item $\overline{[a, bc, c^{-1}]}(x_0\otimes x_1\otimes x_2\otimes x_3) = [a, b, c](\I^{-1}_{b, c, (bc)^{-1}}(x_0)\otimes \I_{abc, c^{-1}, (ab)^{-1}}(x_1)\otimes x_3\otimes x_2)$ for any $x_0\in V^{*}_{bc, c^{-1}, b^{-1}}$, $x_1\in V_{abc, c^{-1}, (ab)^{-1}}$, $x_2\in V^{*}_{a, b, (ab)^{-1}}$ and $x_3\in V_{a, bc, (abc)^{-1}}$.
\end{enumerate}
\end{definition}

Describe the geometric meaning of the conditions from the definition \ref{Definition:GSystem}. Let $T^{+}$ be a positively oriented tetrahedron with vertices $0, 1, 2, 3$.  As before, orient each edge $[i, j]$ of $T^{+}$ from $i$ to $j$ for $i < j$, $i, j\in\{0, 1, 2, 3\}$. Colour the oriented edges so that the chain $0 - 1 - 2 - 3$ is coloured by the elements $a, b, c\in G$ (figure \ref{Figure:ReorderTetrahedron} on the left). As described earlier, to the faces $\partial_0 T^{+}$, $\partial_1 T^{+}$, $\partial_2 T^{+}$ and $\partial_3 T^{+}$ are assigned modules $V_{b, c, (bc)^{-1}}$, $V^{*}_{ab, c, (abc)^{-1}}$, $V_{a, bc, (abc)^{-1}}$ and $V^{*}_{a, b, (ab)^{-1}}$ respectively. The tetrahedron $T^{+}$ corresponds to the tensor $[a, b, c]$.

\begin{figure}[h]
\begin{center}
\ \hfill
\begin{tikzpicture}[scale=1.2]
\begin{scope}[line, decoration={
    markings,
    mark=at position 0.5 with {\arrow{latex}}}
    ] 
    \draw[postaction={decorate}] (3, 5.2) -- (0, 0) node[left, midway] {$b$};
    \draw[postaction={decorate}] (3, 5.2) -- (6, 0) node[above right, midway] {$bc$};
    \draw[postaction={decorate}] (0, 0) -- (6, 0) node[below right, midway] {$c$};
    \draw[postaction={decorate}] (3, 1.73) -- (3, 5.2) node[above right, midway] {$a$};
    \draw[postaction={decorate}] (3, 1.73) -- (0, 0) node[above left, midway] {$ab$};
    \draw[postaction={decorate}] (3, 1.73) -- (6, 0) node[above right, midway] {$abc$};
\end{scope}
	\draw[fill] (0, 0) node[left] {$2$} circle (0.035);
	\draw[fill] (6, 0) node[right] {$3$} circle (0.035);
	\draw[fill] (3, 5.2) node[above] {$1$} circle (0.035);
	\draw[fill] (3, 1.73) node[below] {$0$} circle (0.035);
	
	\draw (2.2, 2) node {$V^{*}_{a, b, (ab)^{-1}}$};
	\draw (3.9, 2.1) node {$V_{a, bc, (abc)^{-1}}$};
	\draw (3, 0.7) node {$V^{*}_{ab, c, (abc)^{-1}}$};
	\draw (5, 4) node {$V_{b, c, (bc)^{-1}}$};
\end{tikzpicture}
\hfill
\begin{tikzpicture}[scale=1.2]
\begin{scope}[line, decoration={
    markings,
    mark=at position 0.5 with {\arrow{latex}}}
    ] 
    \draw[postaction={decorate}] (3, 5.2) -- (0, 0) node[left, midway] {$b$};
    \draw[postaction={decorate}] (3, 5.2) -- (6, 0) node[above right, midway] {$bc$};
    \draw[postaction={decorate}] (0, 0) -- (6, 0) node[below right, midway] {$c$};
    \draw[postaction={decorate}] (3, 5.2) -- (3, 1.73) node[above right, midway] {$a^{-1}$};
    \draw[postaction={decorate}] (3, 1.73) -- (0, 0) node[above left, midway] {$ab$};
    \draw[postaction={decorate}] (3, 1.73) -- (6, 0) node[above right, midway] {$abc$};
\end{scope}
	\draw[fill] (0, 0) node[left] {$2$} circle (0.035);
	\draw[fill] (6, 0) node[right] {$3$} circle (0.035);
	\draw[fill] (3, 5.2) node[above] {$0$} circle (0.035);
	\draw[fill] (3, 1.73) node[below] {$1$} circle (0.035);
	
	\draw (2.2, 2) node {$V_{a^{-1}, ab, b^{-1}}$};
	\draw (4.1, 1.8) node {$V^{*}_{a^{-1}, abc, (bc)^{-1}}$};
	\draw (3, 0.7) node {$V^{*}_{ab, c, (abc)^{-1}}$};
	\draw (5, 4) node {$V_{b, c, (bc)^{-1}}$};
\end{tikzpicture}
\hfill\ \ 
\end{center}
\caption{\label{Figure:ReorderTetrahedron}Faces and corresponding modules for positively oriented tetrahedrons (on the left) and for negatively oriented tetrahedrons obtained by swapping $0$ and $1$ (on the right)}
\end{figure}
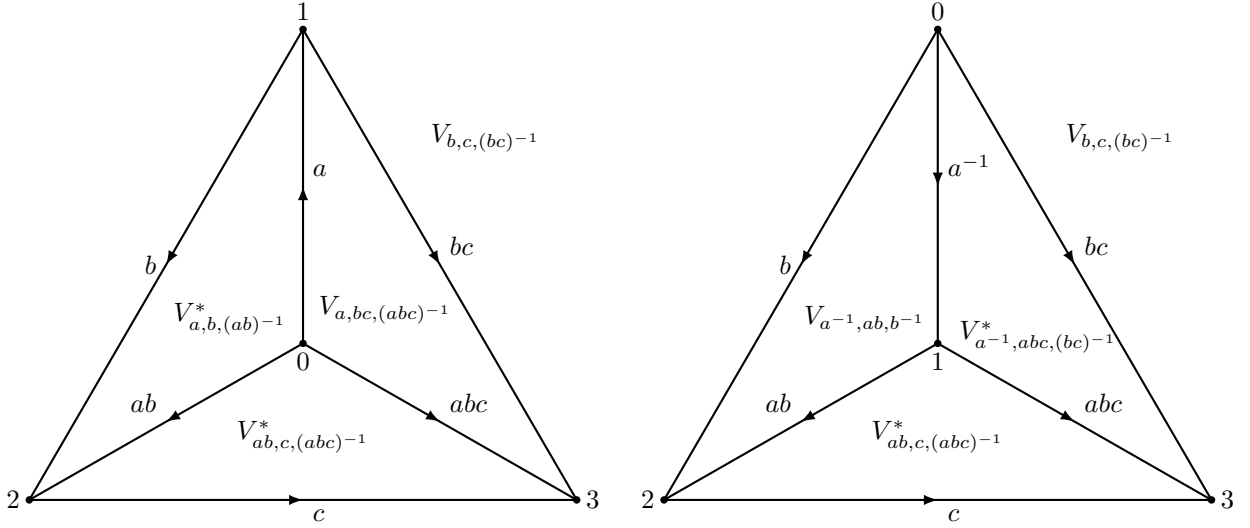

Swap the numbers $0$ and $1$ in the vertex numbering, but keep all the edge colours the same. Note that the edge connecting $0$ and $1$ is flipped, so we change the colour $a$ to $a^{-1}$ (figure \ref{Figure:ReorderTetrahedron} on the right). We get the negatively oriented tetrahedron $T^{-}$, which corresponds to the tensor $\overline{[a^{-1}, ab, c]}$. The face $\partial_0 T^{-}$ is the face $\partial_1 T^{+}$ with the associated module $V^{*}_{ab, c, (abc)^{-1}}$, the face $\partial_1 T^{-}$ is the face $\partial_0 T^{+}$ with the associated module $V_{b, c, (bc)^{-1}}$, the face $\partial_2 T^{-}$ is the face $\partial_2 T^{+}$ with associated module $V^{*}_{a^{-1}, abc, (bc)^{-1}}$, which is isomorphic to $\I_{a, bc, (abc)^{-1}}(V_{a, bc, (abc)^{-1}})$, and the face $\partial_3 T^{-}$ is the face $\partial_3 T^{+}$ with associated module $V_{a^{-1}, ab, b^{-1}}$, which is isomorphic to $\I^{-1}_{a^{-1}, ab, b^{-1}}(V^{*}_{a, b, (ab)^{-1}})$. So the first condition of the definition \ref{Definition:GSystem} guarantees that the tensors $[a, b, c]$ and $\overline{[a^{-1}, bc, c]}$ are the same. The second conditions come from the similar construction, but when the tetrahedron $T^{-}$ is obtained by swapping vertices $1$ and $2$. The third condition corresponds to swapping vertices $2$ and $3$. In other words, in the $G$-system the tensors associated with the coloured tetrahedron do not depend on the order of the vertices.

Recall the operation of the tensor \emph{contraction}. Let $\varphi\colon A\otimes B\to \K$ and $\psi\colon A^{*}\otimes C\to \K$ be tensors. Then the result of the contraction of $\varphi$ and $\psi$ along the module $A$ is a tensor $\zeta = *_{A}(\varphi\otimes \psi)$, $\zeta\colon B\otimes C\to\K$, obtained as follows. Choose any basis $\{e_1, \ldots, e_n\}$ of the module $A$, and let $\{e^1, \ldots, e^n\}$ be a dual basis of the dual module $A^{*}$. Then for any $b\in B$ and $c\in C$: $$\zeta(b\otimes c) = \sum\limits_{i = 1}^{n}\varphi(e_i\otimes b)\cdot \psi(e^i\otimes c).$$

It's easy to check that this definition of the contraction operation is correct (i.e. the result does not depend on the chosen basis of the module $A$). The tensor contraction operation naturally extends to contraction along multiple modules. For example, if $\varphi\colon A_1\otimes A_2\otimes B_1\otimes B_2\to \K$ and $\psi\colon A_1^{*}\otimes C_1\otimes C_{2}\otimes B_1^{*}\to \K$, then $*_{A_1, B_1}(\varphi\otimes \psi)\colon A_2\otimes B_2\otimes C_1\otimes C_2\to \K$ defined by the formula $$*_{A_1, B_1}(\varphi\otimes \psi)(a\otimes b\otimes c_1\otimes c_2) = \sum\limits_{i = 1}^{n}\sum\limits_{j = 1}^{m}\varphi(e_i\otimes a\otimes f_j\otimes b)\cdot \psi(e^i\otimes c_1\otimes c_2\otimes f^j),$$ where $\{e_1, \ldots, e_n\}$ is a basis of $A_1$, $\{e^1, \ldots, e^n\}$ is a dual basis of $A_1^{*}$, $\{f_1, \ldots, f_m\}$ is a basis of $B_1$, $\{f^1, \ldots, f^m\}$ is a dual basis of $B_1^{*}$.

\begin{definition}
\label{Definition:SpecialGSystem}
Let $G$ be a finite group. A $G$-system $(\{V_{a, b, c}\}, \{\I_{a, b, c}\}, \{[a, b, c]\}, \{\overline{[a, b, c]}\})$ is called \emph{special} if for any $a, b, c, d\in G$: $$
*_{V_{ab, cd, (abcd)^{-1}}}([ab, c, d]\otimes [a, b, cd]) = *_{V_{b, c, (bc)^{-1}}, V_{bc, d, (bcd)^{-1}}, V_{a, bc, (abc)^{-1}}}([b, c, d]\otimes [a, bc, d]\otimes [a, b, c]).
$$
\end{definition}

Both sides of the condition in the definition \ref{Definition:SpecialGSystem} are tensors on $$V_{b, cd, (bcd)^{-1}}\otimes V_{a, bcd, (abcd)^{-1}}\otimes V^{*}_{a, b, (ab)^{-1}}\otimes V_{c, d, (cd)^{-1}}\otimes V^{*}_{abc, d, (abcd)^{-1}}\otimes V^{*}_{ab, c, (abc)^{-1}}.$$

\begin{definition}
\label{Definition:StrongSpecialGSystem}
Let $G$ be a finite group. A special $G$-system $(\{V_{a, b, c}\}, \{\I_{a, b, c}\}, \{[a, b, c]\}, \{\overline{[a, b, c]}\})$ is called \emph{strong} if, for any $a, b, c\in G$, it satisfies the following additional condition: for any tensors $A\colon V^{*}_{ab, c, (abc)^{-1}}\otimes X\to \K$, $B\colon V_{ab, c, (abc)^{-1}}\otimes Y\to \K$, $C\colon V^{*}_{a, bc, (abc)^{-1}}\otimes Z\to \K$ and $D\colon V_{a, bc, (abc)^{-1}}\otimes T\to \K$ 
\begin{multline*}
*_{V_{ab, c, (abc)^{-1}}, V_{a, bc, (abc)^{-1}}} (A\otimes B\otimes C\otimes D) = \\ = *_{V_{ab, c, (abc)^{-1}}, V_{ab, c, (abc)^{-1}}, V_{a, bc, (abc)^{-1}}, V_{a, bc, (abc)^{-1}}} (A\otimes B\otimes *_{V_{a, b, (ab)^{-1}}, V_{b, c, (bc)^{-1}}}([a, b, c]\otimes \overline{[a, b, c]})\otimes C\otimes D),
\end{multline*}
where $X, Y, Z, T$ are tensor products of modules from the family $\{V_{a, b, c}\}$ and it duals.
\end{definition}

Let $P$ be an oriented simple spine of the oriented manifold $M$. Fix an arbitrary true vertex $v\in\mathcal{V}(P)$. The neighbourhood of $v$ in $P$ divides the neighbourhood of $v$ in $M$ into four balls. Enumerate them in arbitrary order by $0, 1, 2, 3$ and denote them by $B_0, B_1, B_2, B_3$. This enumeration induces orientations of cells and edges in the neighbourhood of $v$ in $P$ as follows. Choose the orientation of the cell that separates the ball $B_i$ from the ball $B_j$, $i < j$, so that with the normal oriented from $B_i$ to $B_j$ it defines the orientation of $M$. Edges incident to $B_0, B_1, B_2$ and $B_0, B_2, B_3$ orient from $v$, the other two edges orient to $v$ (figure \ref{Figure:VertexNeighbourhood}). We will refer to these orientations as \emph{local orientations} of cells and edges in the neighbourhood of $v$ in $P$. Note that local orientations do not necessarily coincide with the orientations of the triple lines and 2-components of $P$ (which are \emph{global orientations}).

\begin{figure}[h]
\begin{center}
\begin{tikzpicture}[scale=0.65]
	\draw[line] (0, 0) -- (5, 1) -- (6, 5);
	\draw[line] (0, 0) -- (1, 4);
	\draw[fill] (3, 2.5) node[above] {$v$} circle (0.05);
	\draw[line] (0, 0) -- (0, -2) -- (3.2, 0.66);
	\draw[line_dashed] (3.2, 0.66)-- (5.4, 2.5);
	\draw[line] (5.4, 2.5) -- (6, 3) -- (6, 5);
	\draw[line] (1, 4) -- (1, 6) -- (5, 3) -- (5, 1);
	\draw[line_dashed] (1, 4) -- (3.1, 4.42);
	\draw[line] (3.1, 4.42) -- (6, 5);
	\draw[line_dashed] (3, 2.5) -- (4.3, 3.58) node[right, pos=0.5] {$e_0$};
\begin{scope}[line, decoration={
    markings,
    mark=at position 0.5 with {\arrow{latex}}}
    ] 
    \draw[postaction={decorate}] (3, 2.5) -- (0, 0) node[left, pos=0.5] {$e_1$};
    \draw[postaction={decorate}] (3, 2.5) -- (1, 4) node[left, pos=0.5] {$e_3$};
\end{scope}

\begin{scope}[line, decoration={
    markings,
    mark=at position 0.75 with {\arrow{latex}}}
    ]
    \draw[postaction={decorate}] (5, 1) -- (3, 2.5) node[right, pos=0.65] {$e_2$};
    \draw[postaction={decorate}] (6, 5) -- (4.3, 3.58);
\end{scope}

	\draw[orientation_line_opp] (1.25, 4.5) arc (-180:90:0.5);
	\draw (1.75, 4.5) node {$a$};
	
	\draw[orientation_line_opp] (1.25, 2.4) arc (-180:90:0.5);
	\draw (1.75, 2.4) node {$ab$};
	
	\draw[orientation_line_opp] (2.5, 1.5) arc (-180:90:0.5);
	\draw (3.0, 1.5) node {$abc$};
	
	\draw[orientation_line] (0.25, -0.5) arc (-180:90:0.5);
	\draw (0.75, -0.5) node {$c$};
	
	\draw[orientation_line_opp] (3.9, 4.2) arc (-180:90:0.4);
	\draw (4.3, 4.2) node {$b$};
	
	\draw[orientation_line_opp] (4.7, 3.6) arc (-180:90:0.4);
	\draw (5.1, 3.6) node {$bc$};
	
	\draw (1.75, 6) node {$B_1$};
	\draw (0.5, 5.5) node {$B_0$};
	\draw (-1.0, -1.5) node {$B_2$};
	\draw (1.5, -1.5) node {$B_3$};
\end{tikzpicture}
\end{center}
\caption{\label{Figure:VertexNeighbourhood}Local orientations around the true vertex $v$}
\end{figure}
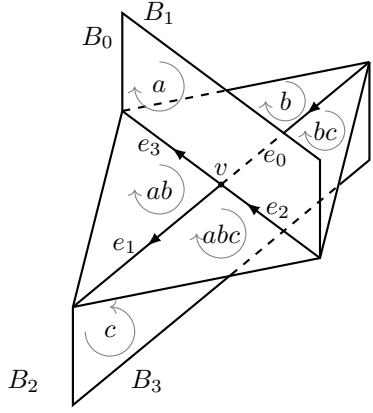

Let $G$ be a finite group and let $\xi\colon \mathcal{C}(P)\to G$ be a colouring of $P$. This colouring induces the local colouring of cells in the neighbourhood of $v$ in $P$ as follows. If the local orientation of the cell coincides with the global orientation of that cell, then the local colour coincides with the global colour. If the orientations are different, then the local colour is inverse to the global colour. Let $a$ be a local colour of the cell incident to $B_0$ and $B_1$, $b$ be a local colour of the cell incident to $B_1$ and $B_2$, and $c$ be a local colour of the cell incident to $B_2$ and $B_3$. The local colour of the cell incident to $B_0$ and $B_2$ is $ab$, incident to $B_0$ and $B_3$ is $abc$, and incident to $B_1$ and $B_3$ is $bc$ (figure \ref{Figure:VertexNeighbourhood}).

Let $(\{V_{a, b, c}\}, \{\I_{a, b, c}\}, \{[a, b, c]\}, \{\overline{[a, b, c]}\})$ be a $G$-system. If the order of the balls $B_0, \ldots, B_3$ agreed with the orientation of $M$, then assign the tensor $[a, b, c]$ to the true vertex $v$. In the opposite case assign the tensor $\overline{[a, b, c]}$. For simplicity, assume that the order of the balls agreed with the orientation of $M$. Then we assign the module $V_{b, c, (bc)^{-1}}$ to the edge $e_0$ incident to $B_1, B_2, B_3$, and the module $V_{a, bc, (abc)^{-1}}$ to the edge $e_2$ incident to $B_0, B_1, B_3$. To the edge $e_1$ incident to $B_0, B_2, B_3$ we assign the dual module $V^{*}_{ab, c, (abc)^{-1}}$, and to the edge $e_3$ incident to $B_0, B_1, B_2$ we assign the dual module $V^{*}_{a, b, (ab)^{-1}}$.

Consider the map $i_0\otimes i_1\otimes i_2\otimes i_3$, where $i_j$ is identity if the local orientation of the edge $e_j$ coincides with the global orientation of the triple line, and $i_j$ is an isomorphism from the family $\{\I_{a, b, c}\}$ (or its inverse) if the local and global orientations are different. Define the tensor $[(v, \xi)]$ by the formula $$[(v, \xi)](x\otimes y\otimes z\otimes t) = [a, b, c](i_0\otimes i_1\otimes i_2\otimes i_3(x\otimes y\otimes z\otimes t)).$$

For example, if all global orientations of triple lines start from the true vertex $v$, then $i_0 = \I^{-1}_{b, c, (bc)^{-1}}$, $i_1 = id_{V^{*}_{ab, c, (abc)^{-1}}}$, $i_2 = \I^{-1}_{a, bc, (abc)^{-1}}$ and $i_3 = id_{V^{*}_{a, b, (ab)^{-1}}}$. So in this case $$[(v, \xi)]\colon V^{*}_{b, c, (bc)^{-1}}\otimes V^{*}_{ab, c, (abc)^{-1}}\otimes V^{*}_{a, bc, (abc)^{-1}}\otimes V^{*}_{a, b, (ab)^{-1}} \to \K$$ and the value of this tensor is defined by the formula $$[(v, \xi)](x\otimes y\otimes z\otimes t) = [a, b, c](\I^{-1}_{b, c, (bc)^{-1}} (x)\otimes y\otimes \I^{-1}_{a, bc, (abc)^{-1}}(z)\otimes t).$$

\begin{definition}
\label{Definition:PolyhedronWeight}
Let $G$ be a finite group, $(\{V_{a, b, c}\}, \{\I_{a, b, c}\}, \{[a, b, c]\}, \{\overline{[a, b, c]}\})$ be a $G$-system, $P$ is an oriented special spine of the oriented 3-manifold $M$, and let $\xi\colon \mathcal{C}(P)\to G$ be a colouring.  Then the result of the contraction of all tensors $[(v, \xi)]$ for all true vertices $v\in \mathcal{V}(P)$ along modules corresponding to triple lines of $P$ is called the \emph{weight} of the coloured polyhedron $P$ and is denoted by $w_{\xi}(P)$.
\end{definition}

\begin{lemma}
\label{Lemma:WeightInvariance}
Let $G$ be a finite group, $(\{V_{a, b, c}\}, \{\I_{a, b, c}\}, \{[a, b, c]\}, \{\overline{[a, b, c]}\})$ be a $G$-system. Let $P$ be an oriented simple spine of the oriented 3-manifold $M$, $\xi\colon \mathcal{C}(P)\to G$ a colouring of $P$. Then
\begin{enumerate}
\item The value $w_{\xi}(P)$ is correctly defined, i.e. it does not depend on the order of the balls in the neighbourhood of each true vertex of the polyhedron $P$;
\item If the oriented special polyhedron $P'$ is obtained from $P$ by changing the orientation of a triple line, then $w_{\xi}(P) = w_{\xi}(P')$.
\end{enumerate}
\end{lemma}
\begin{proof}
The first statement follows from the definition \ref{Definition:GSystem}. Conditions in this definition guarantee that the tensor associated with the true vertex does not change when the balls in the neighbourhood of this vertex are re-enumerated.

To prove the second statement, consider two true vertices $v_1, v_2$ of the triple line with different orientations in $P$ and $P'$. The tensor $[(v_1, \xi)]$ for $P$ and $P'$ differs by only one component (an isomorphism $\I_{a, b, c}$ is added or removed). Similarly, the tensor $[v_2, \xi]$ for $P$ and $P'$ also differs by one component. Moreover, if the isomorphism is added to the component of $[v_1, \xi]$, then it is removed from the component of $[v_2, \xi]$ and vice versa. So the result of the contraction along the module corresponding to this triple line remains the same.
\end{proof}

\subsection{Invariant $DW_{\Sy}$}

Denote by $\Z\K$ the set of all finite formal sums of elements of $\K$ with integer coefficients (i.e. $\Z\K = \{z_1 k_1 + \ldots + z_n k_n | z_i\in\Z, k_i\in \K\}$). There are natural operations ''$+$`` and ''$\cdot$`` on this set. To sum two formal sums, we should simply write the sums together and replace equal elements of $\K$ on a summand with the sum of the coefficients. To multiply two formal sums, we should open the parentheses as for polynomials and also replace similar summands with one summand. Note that these operations do not use summation in the ring $\K$, only multiplication. It's clear that $\Z\K$ is a ring.

Let $X$ be a finite multiset over $\K$ (i.e. the finite collection of elements of $\K$, possible with repetitions). For each $x\in X$, denote by $|X|_{x}$ the number of elements $x$ in the multiset $X$. Define the map $\Zmap$ that maps $X$ to the element of $\Z\K$ by the following rule: $$\Zmap(X) = |X|_{x_1} x_1 + \ldots + |X|_{x_n} x_n,$$ where $x_1, \ldots, x_n$ are all different elements of $X$.

\begin{definition}
\label{Definition:DWInvariant}
Let $G$ be a finite group, $\Sy = (\{V_{a, b, c}\}, \{\I_{a, b, c}\}, \{[a, b, c]\}, \{\overline{[a, b, c]}\})$ be a strong special $G$-system, $P$ be an oriented special spine of the oriented 3-manifold $M$. Then the \emph{symmetric Dijkgraaf -- Witten type invariant} of the manifold $M$ is $$DW_{\Sy}(M) = \Zmap(\{w_{\xi}(P) | \xi\in Col_{G}(P)\}).$$
\end{definition}

\begin{remark}
\label{Remark:DW}
To find the value of the invariant $DW_{\Sy}$ for the oriented manifold $M$, we should fix a special spine $P$ of $M$, and then fix any orientation of $P$. Then enumerate all colourings of $P$ by elements of the finite group $G$. For each colouring compute the value $w_{\xi}(P)\in \K$. As a result we get the multiset $\{x_{\xi_1}(P), \ldots, w_{\xi_n}(P)\}$. Finally, we should write this multiset as a formal sum with integer coefficients.
\end{remark}

\begin{theorem}
\label{Theorem:DWInvariance}
Let $G$ be a finite group, $\Sy = (\{V_{a, b, c}\}, \{\I_{a, b, c}\}, \{[a, b, c]\}, \{\overline{[a, b, c]}\})$ be a strong special $G$-system. Then the value $DW_{\Sy}(M)$ for an oriented 3-manifold $M$ does not depend on the special spine of $M$ or on the orientation of this special spine.
\end{theorem}
\begin{proof}
Let $P$ be a special spine of $M$. For any colouring $\xi\in Col_{G}(P)$ by the lemma \ref{Lemma:WeightInvariance} the value $w_{\xi}(P)$ does not depend on the orientation of the triple lines of $P$. The whole multiset $\{w_{\xi}(P) | \xi\in Col_{G}(P)\}$ is not changed by reversing the orientation of any 2-component of $P$ (this reversal only reorders the values in the multiset).

We should prove that if $P_1$ and $P_2$ are two simple spines of $M$, then the multisets $\{w_{\xi}(P_1) | \xi\in Col_{G}(P_1)\}$ and $\{w_{\xi}(P_2) | \xi\in Col_{G}(P_2)\}$ are the same.

Let $P_2$ be obtained from $P_1$ by a $T$-move, and let $\xi'$ be the colouring of $P_2$, corresponding to the colouring $\xi$ of $P_1$. Let $A, B$ be the true vertices of $P_1$ before the move, and $X, Y, Z$ be the true vertices of $P_2$ that appear after the $T$-move. Choose the orientations of the 2-components and triple lines of $P_1$ and $P_2$, and the colours of the 2-components of $P_1$ and $P_2$ as shown in the figure \ref{Figure:DWTMove}. Note that in this figure the local and global orientations coincide.

\begin{figure}[h]
\begin{center}
\ \hfill
\begin{tikzpicture}[scale=0.55]

	\draw[line] (2, 7) -- (0, 7) -- (0, 0) -- (2, 0) -- (4.5, 2.5) -- (7, 0) -- (9, 0) -- (9, 7) -- (7, 7) -- (4.5, 4.5) -- cycle;
	\draw[line] (2, 7) arc (-180:0:2.5 and 0.75);
	\draw[line] (7, 7) arc (0:180:2.5 and 0.75);
	
	\draw[line] (2, 0) arc (-180:0:2.5 and 0.75);
	\draw[line_dashed] (7, 0) arc (0:180:2.5 and 0.75);
	
	\draw[line] (4.5, 2.5) -- (4.5, 4.5);
	
	\draw[line] (5, 7.73) -- (7.5, 9) -- (7.5, 7);
	\draw[line_dashed] (7.5, 7) -- (7.5, 2) -- (5, 0.73);
	
	\draw[fill] (4.5, 2.5) node[above right] {$B$} circle (0.05);
	\draw[fill] (4.5, 4.5) node[below left] {$A$} circle (0.05);
	
 	\draw[latex-, thick] (4.5, 3) -- (4.5, 4);
 	\draw[latex-, thick] (4, 5) -- (3, 6);
 	\draw[latex-, thick] (5, 5) -- (6, 6);
 	\draw[latex-, thick] (3, 1) -- (4, 2);
 	\draw[-latex, thick] (6, 1) -- (5, 2);
 	\draw[-latex, thick, dashed] (4.5, 4.5) -- (5, 7.73);
 	\draw[latex-, thick, dashed] (5, 0.73) -- (4.5, 2.5);
	
	\draw[orientation_line] (0.3, 6) arc (-180:90:0.7);
	\draw (1.0, 6.0) node {$abcd$};
	
	\draw[orientation_line] (7.5, 1) arc (-180:90:0.5);
	\draw (8.0, 1.0) node {$cd$};
	
	\draw[orientation_line] (6.3, 8) arc (-180:90:0.5);
	\draw (6.8, 8.0) node {$ab$};
	
	\draw[orientation_line_opp] (5.2, 7.0) arc (-180:90:0.5);
	\draw (5.7, 7.0) node {$b$};
	
	\draw[orientation_line] (3.0, 7.0) arc (-180:90:0.5);
	\draw (3.5, 7.0) node {$a$};
	
	\draw[orientation_line] (3.9, 5.5) arc (-180:90:0.5);
	\draw (4.4, 5.5) node {$bcd$};
	
	\draw[orientation_line] (3.8, 1.5) arc (-180:90:0.4);
	\draw (4.2, 1.5) node {$abc$};
	
	\draw[orientation_line] (4.0, -0.15) arc (-180:90:0.5);
	\draw (4.5, -0.15) node {$d$};
	
	\draw[orientation_line] (5.0, 1.15) arc (-180:90:0.3);
	\draw (5.3, 1.15) node {$c$};
\end{tikzpicture}
\hfill
\begin{tikzpicture}[scale=0.55]
	
	\draw[line] (0, 0) -- (2, 0) -- (2, 7) -- (0, 7) -- cycle;
	\draw[line] (7, 0) -- (9, 0) -- (9, 7) -- (7, 7) -- cycle;
	\draw[line] (5, 7.73) -- (7.5, 9) -- (7.5, 7);
	\draw[line_dashed] (7.5, 7) -- (7.5, 2) -- (5, 0.73);
	\draw[latex-, thick, dashed] (5, 0.73) -- (5, 4.25);
	\draw[-latex, thick, dashed] (5, 4.25) -- (5, 7.73);
	
	\draw[line] (2, 7) arc (-180:0:2.5 and 0.75);
	\draw[line] (7, 7) arc (0:180:2.5 and 0.75);
	
	\draw[line] (2, 0) arc (-180:0:2.5 and 0.75);
	\draw[line_dashed] (7, 0) arc (0:180:2.5 and 0.75);
	\draw[line] (2, 3.5) arc (-180:0:2.5 and 0.75);
	\draw[line_dashed, -latex] (7, 3.5) arc (0:78:2.5 and 0.75);
	\draw[line_dashed] (5, 4.25) arc (90:125:3.0 and 0.55);
	\draw[line_dashed, -latex] (2, 3.5) arc (180:115:2.5 and 0.75);
	
	\draw[fill] (2, 3.5) node[left] {$Y$} circle (0.05);
	\draw[fill] (7, 3.5) node[left] {$X$} circle (0.05);
	\draw[fill] (5, 4.25) node[above right] {$Z$} circle (0.05);
	
	\draw[latex-, thick] (2, 1) -- (2, 3);
	\draw[latex-, thick] (2, 4) -- (2, 6);
	\draw[-latex, thick] (7, 1) -- (7, 3);
	\draw[latex-, thick] (7, 4) -- (7, 6);
	\draw[latex-, thick] (4.5, 2.75) arc (-90:-45:2.5 and 0.75);
	
	\draw[orientation_line] (0.3, 6) arc (-180:90:0.7);
	\draw (1.0, 6.0) node {$abcd$};
	
	\draw[orientation_line] (7.5, 1) arc (-180:90:0.5);
	\draw (8.0, 1.0) node {$cd$};
	
	\draw[orientation_line] (6.3, 8) arc (-180:90:0.5);
	\draw (6.8, 8.0) node {$ab$};
	
	\draw[orientation_line_opp] (5.2, 7.0) arc (-180:90:0.5);
	\draw (5.7, 7.0) node {$b$};
	
	\draw[orientation_line] (3.0, 7.0) arc (-180:90:0.5);
	\draw (3.5, 7.0) node {$a$};
	
	\draw[orientation_line] (3.9, 5.5) arc (-180:90:0.5);
	\draw (4.4, 5.5) node {$bcd$};
	
	\draw[orientation_line] (2.5, 1.5) arc (-180:90:0.5);
	\draw (3, 1.5) node {$abc$};
	
	\draw[orientation_line] (4.0, -0.15) arc (-180:90:0.5);
	\draw (4.5, -0.15) node {$d$};
	
	\draw[orientation_line] (5.5, 1.7) arc (-180:90:0.5);
	\draw (6, 1.7) node {$c$};
	
	\draw[orientation_line_opp] (3.7, 3.5) arc (-180:90:0.5);
	\draw (4.2, 3.5) node {$bc$};
\end{tikzpicture}
\hfill \ \ 
\end{center}
\caption{\label{Figure:DWTMove}Local orientations and colourings before (on the left) and after (on the right) $T$-move}
\end{figure}
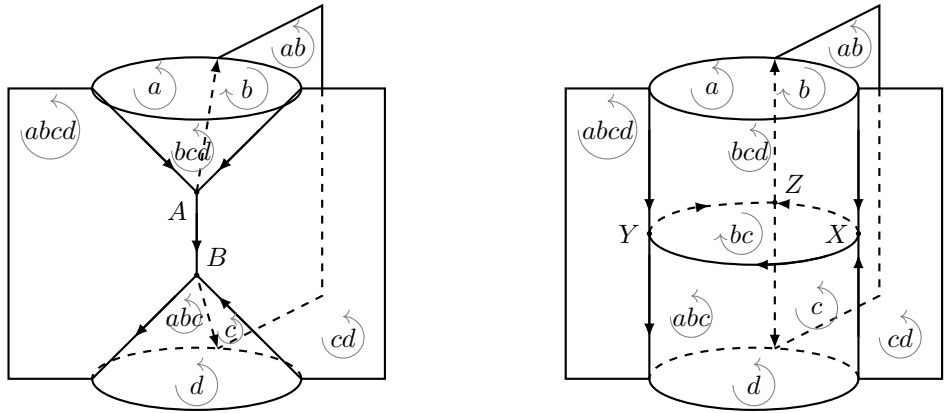

Then $$*_{V_{ab, cd, (abcd)^{-1}}}([A, \xi]\otimes [(B, \xi)]) = *_{V_{ab, cd, (abcd)^{-1}}}([a, b, cd]\otimes [ab, c, d]),$$
\begin{multline*}
*_{V_{b, c, (bc)^{-1}}, V_{bc, d, (bcd)^{-1}}, V_{a, bc, (abc)^{-1}}}([(X, \xi')]\otimes [(Y, \xi')]\otimes [(Z, \xi')]) = \\ = *_{V_{b, c, (bc)^{-1}}, V_{bc, d, (bcd)^{-1}}, V_{a, bc, (abc)^{-1}}} ([b, c, d]\otimes [a, bc, d]\otimes [a, b, c]).
\end{multline*}

Our $G$-system $\Sy$ is special, so, these tensors are the same. Hence $w_{\xi}(P_1) = w_{\xi'}(P_2)$.

For the case where $P_2$ is obtained from $P_1$ by an $L$-move the proof is similar. Let $X, Y$ be true vertices in $P_2$ that appear after the $L$-move. We can choose the orientations and colours of $P_1$ and $P_2$ as shown in the figure \ref{Figure:DWLMove}. Then $$*_{V_{b, c, bc^{-1}}, V_{a, b, (ab)^{-1}}} ([X, \xi']\otimes [Y, \xi']) = *_{V_{b, c, bc^{-1}}, V_{a, b, (ab)^{-1}}} ([a, b, c]\otimes \overline{[a, b, c]}).$$

\begin{figure}[h]
\begin{center}
\ \hfill
\begin{tikzpicture}[scale=0.55]

	\draw[line] (2, 0) -- (0, 0) -- (0, 7) -- (2, 7);
	\draw[line] (7, 0) -- (9, 0) -- (9, 7) -- (7, 7);
	\draw[line] (2, 7) arc (-180:0:2.5 and 0.75);
	\draw[line] (7, 7) arc (0:180:2.5 and 0.75);
	
	\draw[line] (2, 0) arc (-180:0:2.5 and 0.75);
	\draw[line_dashed] (7, 0) arc (0:180:2.5 and 0.75);
	
	\draw[line] (2, 0) arc (180:90: 2.5);
	\draw[line, -latex] (7, 0) arc (0:94: 2.5);
	\draw[line, -latex] (2, 7) arc (-180:-86: 2.5);
	\draw[line] (7, 7) arc (0:-90: 2.5);
	
	\draw[orientation_line] (3.9, 3.5) arc (-180:90:0.6);
	\draw (4.5, 3.5) node {$abc$};
	
	\draw[orientation_line] (4.0, 7.0) arc (-180:90:0.5);
	\draw (4.5, 7.0) node {$a$};
	
	\draw[orientation_line] (4.0, 5.5) arc (-180:90:0.5);
	\draw (4.5, 5.5) node {$bc$};
	
	\draw[orientation_line] (4.0, 1.5) arc (-180:90:0.5);
	\draw (4.5, 1.5) node {$ab$};
	
	\draw[orientation_line] (4.0, -0.15) arc (-180:90:0.5);
	\draw (4.5, -0.15) node {$c$};
\end{tikzpicture}
\hfill
\begin{tikzpicture}[scale=0.55]
	\draw[line] (0, 0) -- (2, 0) -- (2, 7) -- (0, 7) -- cycle;
	\draw[line] (7, 0) -- (9, 0) -- (9, 7) -- (7, 7) -- cycle;
	
	\draw[line] (2, 7) arc (-180:0:2.5 and 0.75);
	\draw[line] (7, 7) arc (0:180:2.5 and 0.75);
	
	\draw[line] (2, 0) arc (-180:0:2.5 and 0.75);
	\draw[line_dashed] (7, 0) arc (0:180:2.5 and 0.75);
	\draw[line] (2, 3.5) arc (-180:0:2.5 and 0.75);
	\draw[line_dashed] (7, 3.5) arc (0:78:2.5 and 0.75);
	\draw[line_dashed, -latex] (2, 3.5) arc (180:78:2.5 and 0.75);
	
	\draw[fill] (2, 3.5) node[left] {$X$} circle (0.05);
	\draw[fill] (7, 3.5) node[right] {$Y$} circle (0.05);
	
	\draw[latex-, thick] (2, 1) -- (2, 3);
	\draw[latex-, thick] (2, 4) -- (2, 6);
	\draw[-latex, thick] (7, 1) -- (7, 3);
	\draw[-latex, thick] (7, 4) -- (7, 6);
	\draw[latex-, thick] (4.5, 2.75) arc (-90:-45:2.5 and 0.75);
	
	\draw[orientation_line] (0.4, 6) arc (-180:90:0.6);
	\draw (1.0, 6.0) node {$abc$};
	
	\draw[orientation_line] (7.4, 6) arc (-180:90:0.6);
	\draw (8.0, 6.0) node {$abc$};
	
	\draw[orientation_line] (4, 7.0) arc (-180:90:0.5);
	\draw (4.5, 7.0) node {$a$};
	
	\draw[orientation_line] (4.0, 5.5) arc (-180:90:0.5);
	\draw (4.5, 5.5) node {$bc$};
		
	\draw[orientation_line] (4.0, -0.15) arc (-180:90:0.5);
	\draw (4.5, -0.15) node {$c$};
	
	\draw[orientation_line] (4.0, 1.7) arc (-180:90:0.5);
	\draw (4.5, 1.7) node {$ab$};
	
	\draw[orientation_line_opp] (4.0, 3.5) arc (-180:90:0.5);
	\draw (4.5, 3.5) node {$b$};
\end{tikzpicture}
\hfill\ \ 
\end{center}
\caption{\label{Figure:DWLMove}Local orientations and colourings before (on the left) and after (on the right) $L$-move}
\end{figure}
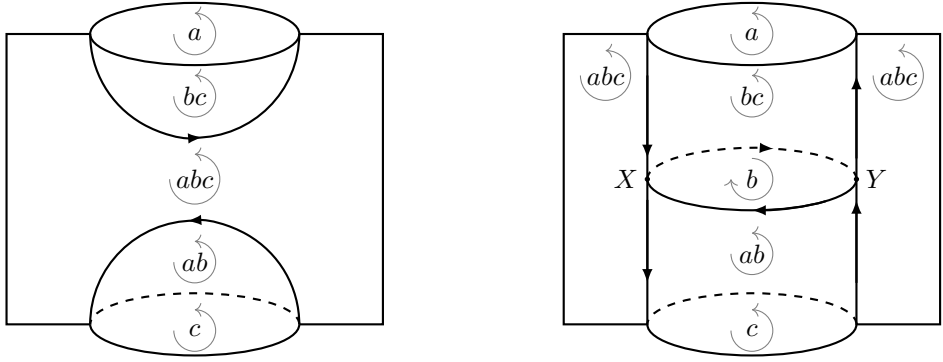

Our $G$-system $\Sy$ is strong, so the results $w_{\xi}(P_1)$ and $w_{\xi'}(P_2)$ of the contractions are the same.
\end{proof}

\begin{remark}
\label{Remark:OnlyT}
In the definitions of the symmetric Dijkgraaf -- Witten type invariants we use the strong special $G$-system $\Sy$. The special property (from the definition \ref{Definition:SpecialGSystem}) guarantees invariance under $T$-moves, and the strong property (from the definition \ref{Definition:StrongSpecialGSystem}) guarantees invariance under $L$-moves. As mentioned earlier (see also \cite[Theorem 1.2.5]{M}), any two special spines of the same 3-manifold with at least two true vertices are related only by $T$-moves. Thus, if we restrict ourselves to 3-manifolds of this class, then to construct an invariant we only need a special $G$-system without strong property.
\end{remark}

\section{1-dimensional case}

\subsection{Chain complex for finite group}

Let $G$ be a finite group, $\ksmall$ a ring. For $n\geqslant 1$ define $$D_n(G; \ksmall) = \left\{\sum\limits_{(g_{i_1}, \ldots, g_{i_n})}\alpha^{i_1\ldots i_n}(g_{i_1}, \ldots, g_{i_n})\right\},$$ where $\alpha^{i_1\ldots i_n}\in \ksmall$, and the sum is taken over all elements $(g_{i_1}, \ldots, g_{i_n})\in G^n$. Elements of the set $D_n(G; \ksmall)$ are finite sequences of elements from $G$ with coefficients from $\ksmall$. The set $D_n(G; \ksmall)$ is an abelian group with component-wise addition. Define $D_n(G; \ksmall)$ to be a trivial group for all $n \leqslant 0$.

For all $n\geqslant 2$, define the boundary homomorphisms $\partial_n\colon D_n(G; \ksmall)\to D_{n - 1}(G; \ksmall)$ as follows:
\begin{multline*}
\partial_n(g_1, \ldots, g_n) = (g_2, \ldots, g_n) - (g_1g_2, g_3, \ldots, g_n) + (g_1, g_2 g_3, g_4, \ldots, g_n) + \ldots + \\ + (-1)^{n - 1}(g_1, \ldots, g_{n-2}, g_{n-1}g_n) + (-1)^n (g_1, \ldots, g_{n - 1})
\end{multline*}
and extend this definition by linearity to all elements of $D_n(G; \ksmall)$. Define $\partial_n = 0$ for all $n\leqslant 1$.

\begin{remark}
\label{Remark:DSimple}
Consider the following maps $d_n^i\colon G^n\to G^{n - 1}$, $i = 0, \ldots, n$, defined as follows: 
\begin{center}
\begin{tabular}{l}
$d_n^{0}(g_1, \ldots, g_n) = (g_2, \ldots, g_n)$, \\
$d_n^i(g_1, \ldots, g_n) = (g_1, \ldots, g_{i - 1}, g_i g_{i + 1}, g_{i + 2}, \ldots, g_n)$ for $1\leqslant i\leqslant n - 1$, \\
$d_n^n(g_1, \ldots, g_n) = (g_1, \ldots, g_{n - 1})$.
\end{tabular}
\end{center}

Then the map $\partial_n\colon D_n(G;\ksmall)\to D_{n - 1} (G; \ksmall)$ for $n\geqslant 2$ can be defined on the basis as follows: $$\partial_n(g_1, \ldots, g_n) = \sum\limits_{i = 0}^{n}(-1)^i d_n^i(g_1, \ldots, g_n).$$
\end{remark}

\begin{theorem}
\label{Theorem:BoundaryHomomorphisms}
For any $g_1, \ldots, g_n\in G$: $\partial_{n - 1}(\partial_n (g_1, \ldots, g_n)) = 0$.
\end{theorem}
\begin{proof}
Using the notations from remark \ref{Remark:DSimple} we can write that $$\partial_{n - 1} (\partial_n (g_1, \ldots, g_n)) = \sum\limits_{i = 0}^{n}\sum\limits_{j = 0}^{n - 1}(-1)^{i + j}d_{n-1}^j(d_n^i(g_1, \ldots, g_n)).$$

Notice that for $j < i$: $d_{n-1}^j(d_n^i(g_1, \ldots, g_n)) = d_{n-1}^{i-1}(d_n^j(g_1, \ldots, g_n))$. Also notice that $$d_{n-1}^i(d_n^{i + 1}(g_1, \ldots, g_n)) = d_{n-1}^i(d_n^{i}(g_1, \ldots, g_n)) = (g_1, \ldots, g_{i - 1}, g_i g_{i+1} g_{i+2}, g_{i + 3}, \ldots, g_n).
$$

So the sum in the image $\partial_{n - 1}(\partial_n (g_1, \ldots, g_n))$ splits into pairs with different signs and equal values. This sum is therefore zero.
\end{proof}

Describe the geometric meaning of the boundary homomorphisms $\partial_n\colon D_n(G;\ksmall)\to D_{n-1}(G;\ksmall)$. Consider the $n$-simplex $T_n = [0, 1, \ldots, n]$ with ordered vertices. Orient each edge $[i, j]$, $i < j$, from $i$ to $j$. Colour the edges $[0, 1]$, $[1, 2]$, $\ldots$, $[n - 1, n]$ by the group elements $g_1, g_2, \ldots, g_n$ respectively. Extend this colouring to other oriented edges by the following rule: if the edge $[i,j]$, $i < j$, is coloured by $x$ and the edge $[j, k]$, $j < k$, is coloured by $y$, then colour the edge $[i,k]$ by element $xy$ (see figure \ref{Figure:Simplex} with example of oriented and coloured simplex $T_4$).

\begin{figure}[h]
\begin{center}
\begin{tikzpicture}[scale=0.65]
\begin{scope}[line, decoration={
    markings,
    mark=at position 0.5 with {\arrow{latex}}}
    ]
	\draw[line, postaction={decorate}] (90:5) -- (162:5) node[above left, midway] {$g_1$};
	\draw[line, postaction={decorate}] (162:5) -- (234:5) node[below left, midway] {$g_2$};
	\draw[line, postaction={decorate}] (234:5) -- (306:5) node[below, midway] {$g_3$};
	\draw[line, postaction={decorate}] (306:5) -- (378:5) node[below right, midway] {$g_4$};
	\draw[line, postaction={decorate}] (90:5) -- (378:5) node[above right, midway] {$g_1 g_2 g_3 g_4$};
	
	\draw[line, postaction={decorate}] (90:5) -- (234:5) node[right, midway] {$g_1 g_2$};
	\draw[line, postaction={decorate}] (90:5) -- (306:5) node[above right, midway] {$g_1 g_2 g_3$};
	\draw[line, postaction={decorate}] (162:5) -- (378:5) node[above, midway] {$g_2 g_3 g_4$};
	
	\draw[line, postaction={decorate}] (162:5) -- (306:5) node[above right, midway] {$g_2 g_3$};
	\draw[line, postaction={decorate}] (234:5) -- (378:5) node[below right, pos=0.45] {$g_3 g_4$};
\end{scope}
	\draw[fill] (90:5) node[above] {0} circle (0.075);
	\draw[fill] (162:5) node[left] {1} circle (0.075);
	\draw[fill] (234:5) node[below left] {2} circle (0.075);
	\draw[fill] (306:5) node[below right] {3} circle (0.075);
	\draw[fill] (378:5) node[right] {4} circle (0.075);
\end{tikzpicture}
\end{center}
\caption{\label{Figure:Simplex}Oriented and coloured simplex $T_4$}
\end{figure}
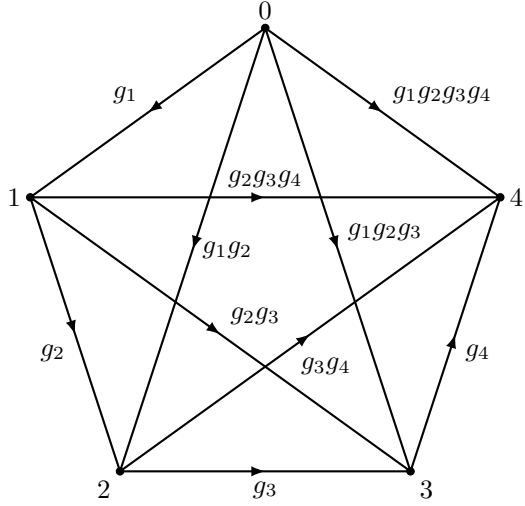

The coloured simplex $T_n$ corresponds to the basis element $(g_1, \ldots, g_n)\in D_n(G; \ksmall)$. The simplex $T_n$ has $n$ faces $\partial_0 T_n = [1, \ldots, n]$, $\partial_1 T_n = [0, 2, \ldots, n]$, $\ldots$, $\partial_n T_n = [0, \ldots, n-1]$. These faces $\partial_0 T_n, \ldots, \partial_n T_n$ are also simplices with oriented and coloured edges (the colouring is induced by the colouring of $T_n$). So each of them corresponds to a basis element of $D_{n - 1}(G;\ksmall)$. In the definition of $\partial_n\colon D_n(G;\ksmall)\to D_{n - 1}(G;\ksmall)$ we map the basis element of $D_n$ (which corresponds to the simplex $T_n$) to the alternative sum of basis elements of $D_{n-1}$ (which corresponds to the faces $\partial_0 T_n, \ldots, \partial_n T_n$).

For $n\geqslant 2$ and for any $g_1, \ldots, g_n\in G$ consider the following elements of $D_n(G;\ksmall)$:
\begin{center}
\begin{tabular}{l}
$\gamma^1_{g_1, \ldots, g_n} = (g_1, \ldots, g_n) + (g_1^{-1}, g_1 g_2, g_3, \ldots, g_n)$, \\
$\gamma^2_{g_1, \ldots, g_n} = (g_1, \ldots, g_n) + (g_1 g_2, g_2^{-1}, g_2 g_3, g_4, \ldots, g_n)$, \\
$\gamma^3_{g_1, \ldots, g_n} = (g_1, \ldots, g_n) + (g_1, g_2 g_3, g_3^{-1}, g_3 g_4, g_5, \ldots, g_n)$, \\
$\vdots$ \\
$\gamma^{n - 1}_{g_1, \ldots, g_n} = (g_1, \ldots, g_n) + (g_1, \ldots, g_{n - 3}, g_{n - 2} g_{n - 1}, g_{n-1}^{-1}, g_{n-1} g_n)$, \\
$\gamma^n_{g_1, \ldots, g_n} = (g_1, \ldots, g_n) + (g_1, \ldots, g_{n - 2}, g_{n-1} g_n, g_n^{-1})$.
\end{tabular}
\end{center}

For $n = 1$ define $\gamma^1_{g_1} = (g_1) + (g_1^{-1})$.

For each $n\geqslant 1$ consider the subgroup $Q_n\subseteq D_n(G;\ksmall)$ generated by the elements $\gamma^1_{g_1, \ldots, g_n}, \ldots, \gamma^n_{g_1, \ldots, g_n}$ for all $g_1, \ldots, g_n\in G$.

\begin{proposition}
\label{Proposition:BoundaryQn}
$\partial_n (Q_n)\subseteq Q_{n - 1}$.
\end{proposition}
\begin{proof}
We should check that for any $i\in\{1, \ldots, n\}$: $\partial_n(\gamma^i_{g_1, \ldots, g_n})\in Q_{n - 1}$.

Initially check equality for $i\neq 1,n$. In this case $$\gamma^i_{g_1, \ldots, g_n} = (g_1, \ldots, g_n) + (g_1, \ldots, g_{i-1}g_i, g_i^{-1}, g_i g_{i + 1}, \ldots, g_n).$$

Notice that $$d_n^{i-1}(g_1, \ldots, g_{i-1}g_i, g_i^{-1}, g_i g_{i + 1}, \ldots, g_n) = d_n^i(g_1, \ldots, g_n) = (g_1, \ldots, g_i, g_i g_{i + 1}, g_{i + 2}, \ldots, g_n),$$ $$d_n^i(g_1, \ldots, g_{i - 1}g_i, g_i^{-1}, g_i g_{i + 1}, \ldots, g_n) = d_n^{i-1}(g_1, \ldots, g_n) = (g_1, \ldots, g_{i-2}, g_{i-1}g_i, g_{i+1}, \ldots, g_n).$$

Also notice that for $j\neq i, i-1$: $$d_n^j(g_1, \ldots, g_n) + d_n^j(g_1, \ldots, g_{i-1}g_i, g_i^{-1}, g_i g_{i + 1}, \ldots, g_n)\in Q_{n-1}.$$

So in the image $\partial_n(\gamma^i_{g_1, \ldots, g_n})$ we cancel the pair of summands $d_n^{i-1}(g_1, \ldots, g_{i-1}g_i, g_i^{-1}, g_i g_{i + 1}, \ldots, g_n)$, $d_n^i(g_1, \ldots, g_n)$ and $d_n^i(g_1, \ldots, g_{i - 1}g_i, g_i^{-1}, g_i g_{i + 1}, \ldots, g_n)$, $d_n^{i-1}(g_1, \ldots, g_n)$ (because these summands have different signs). All other summand pairs are in $Q_{n - 1}$, so $\partial_n(\gamma^i_{g_1, \ldots, g_n})\in Q_{n - 1}$.

For the cases $i = 0, n$ the proof is similar. It's sufficient to note that $$d_n^0(g_1^{-1}, g_1 g_2, \ldots, g_n) = d_n^1(g_1, \ldots, g_n) = (g_1 g_2, g_3, \ldots, g_n),$$ $$d_n^1(g_1^{-1}, g_1 g_2, \ldots, g_n) = d_n^0(g_1, \ldots, g_n) = (g_2, \ldots, g_n),$$ $$d_n^{n-1}(g_1, \ldots, g_{n-1}g_n, g_n^{-1}) = d_n^n(g_1, \ldots, g_n) = (g_1, \ldots, g_{n-1}),$$ $$d_n^n(g_1, \ldots, g_{n-1}g_n, g_n^{-1}) = d_n^{n-1}(g_1, \ldots, g_n) = (g_1, \ldots, g_{n-1} g_n).$$

Then, as previously, $\partial_n(\gamma^1_{g_1, \ldots, g_n}), \partial_n(\gamma^n_{g_1, \ldots, g_n})\in Q_{n - 1}$.
\end{proof}

For the finite group $G$ and the abelian group $\ksmall$, construct the chain complex as follows. Chain groups are $C_n(G; \ksmall) = D_n(G; \ksmall) / Q_n$. Boundary homomorphisms $\partial_n\colon C_n(G; \ksmall)\to C_{n - 1}(G; \ksmall)$ are induced by $\partial_n\colon D_n(G; \ksmall)\to D_{n-1}(G; \ksmall)$ (and denoted by the same symbols). It follows from the theorem \ref{Theorem:BoundaryHomomorphisms} and proposition \ref{Proposition:BoundaryQn} that this is a correctly defined chain complex.

Consider the corresponding cochain complex with cochain groups $C^n(G; \ksmall) = \{\omega\colon C_n(G; \ksmall)\to \ksmall\}$ and coboundary homomorphisms $\partial^n\colon C^{n - 1}(G; \ksmall)\to C^n(G; \ksmall)$. The most important for our purposes are the 3-cocycles $\omega\in \ker \partial^4$. These cocycles are maps $\omega\colon G^3\to \ksmall$ which satisfy the following conditions for any $g_1, g_2, g_3, g_4\in G$:
\begin{enumerate}
\item $\omega(g_1, g_2, g_3) + \omega(g_1^{-1}, g_1 g_2, g_3) = 0$;
\item $\omega(g_1, g_2, g_3) + \omega(g_1 g_2, g_2^{-1}, g_3) = 0$;
\item $\omega(g_1, g_2, g_3) + \omega(g_1, g_2 g_3, g_3^{-1}) = 0$;
\item $\omega(g_2, g_3, g_4) - \omega(g_1 g_2, g_3, g_4) + \omega(g_1, g_2 g_3, g_4) - \omega(g_1 g_2, g_3 g_4) + \omega(g_1, g_2, g_3) = 0$.
\end{enumerate}

\subsection{Invariant $DW_{\omega}$}

For each 3-cocycle $\omega\in\ker\partial^4$, define the special $G$-system $\Sy_{\omega}$ as follows. The ring $\K$ is a ring with multiplicative group isomorphic to additive group of $\ksmall$. All modules $V_{a, b, c}$ are one-dimensional. All isomorphisms $\I_{a, b, c}$ are trivial (this means that $\I_{a, b, c}(1)$ is a linear map from $\K$ to $\K$ such that $1\mapsto 1$). The tensors $[a, b, c]$ and $\overline{[a, b, c]}$ are tensors on a one-dimensional module, so they are just elements of $\K$. For all $a, b, c\in G$ define $[a, b, c] = \omega(a, b, c)$ and $\overline{[a, b, c]} = -\omega(a, b, c)$.

It is clear that $\Sy_{\omega}$ is a special $G$-system, so it defines the symmetric Dijkgraaf -- Witten type invariant, which we'll call $DW_{\omega}$. In this simple one-dimensional case, the additive structure of $\K$ is not important. So it's convenient to write the multiplicative operation on $\K$ additively.

The invariant $DW_{\omega}$ can be computed in the following way. Let $P$ be a special spine of the manifold $M$. Choose any orientation of 2-components and triple lines of $P$, and choose the colouring $\xi\colon P\to G$. Then, for each true vertex $v\in \mathcal{V}(P)$, compute the value $\omega(a, b, c)$, where $a, b, c$ are colours of cells incident to $v$, with the property that exactly one of them (with colour $b$) is incident to others (with colours $a$ and $c$). Get $w_{\xi}(P)$ as the sum of values $\omega(a, b, c)$ over all true vertices. Then $DW_{\omega}(M) = \Zmap(\{w_{\xi}(P) | \xi\in Col_G(P)\})$.

\begin{proposition}
\label{Proposition:DWTrivial}
Let $G$ be a finite group. Then for any $\delta\in C^2(G; \ksmall)$ the invariant $DW_{\omega}$, where $\omega = \partial^3(\delta)$, is trivial.
\end{proposition}
\begin{remark}
\label{Remark:Trivial}
The triviality of $DW_{\omega}$ means that if $P$ is a special polyhedron and $\xi\colon\mathcal{C}(P)\to G$ is colouring, then $w_{\xi}(P) = 0\in\K$.
\end{remark}
\begin{proof}
Let $P$ be a special oriented polyhedron, $\xi\colon \mathcal{C}(P)\to G$ a colouring, and $v\in\mathcal{V}(P)$ a true vertex. Let $B_0, \ldots, B_3$ be balls in the neighbourhood of $v$. Let $a, b, c\in G$ be the colours of the cells separating $B_0$ from $B_1$, $B_1$ from $B_2$ and $B_2$ from $B_3$ respectively (figure \ref{Figure:VertexSimple}). Assume that the order of the balls $B_0, \ldots, B_3$ is consistent with the orientation of $M$.

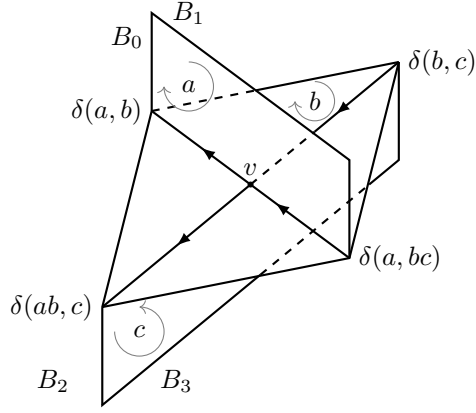
\begin{figure}[h]
\begin{center}
\begin{tikzpicture}[scale=0.65]
	\draw[line] (0, 0) -- (5, 1) -- (6, 5);
	\draw[line] (0, 0) -- (1, 4);
	\draw[fill] (3, 2.5) node[above] {$v$} circle (0.05);
	\draw[line] (0, 0) -- (0, -2) -- (3.2, 0.66);
	\draw[line_dashed] (3.2, 0.66)-- (5.4, 2.5);
	\draw[line] (5.4, 2.5) -- (6, 3) -- (6, 5);
	\draw[line] (1, 4) -- (1, 6) -- (5, 3) -- (5, 1);
	\draw[line_dashed] (1, 4) -- (3.1, 4.42);
	\draw[line] (3.1, 4.42) -- (6, 5);
	\draw[line_dashed] (3, 2.5) -- (4.3, 3.58);
\begin{scope}[line, decoration={
    markings,
    mark=at position 0.5 with {\arrow{latex}}}
    ] 
    \draw[postaction={decorate}] (3, 2.5) -- (0, 0);
    \draw[postaction={decorate}] (3, 2.5) -- (1, 4);
\end{scope}

\begin{scope}[line, decoration={
    markings,
    mark=at position 0.75 with {\arrow{latex}}}
    ]
    \draw[postaction={decorate}] (5, 1) -- (3, 2.5);
    \draw[postaction={decorate}] (6, 5) -- (4.3, 3.58);
\end{scope}

	\draw[orientation_line_opp] (1.25, 4.5) arc (-180:90:0.5);
	\draw (1.75, 4.5) node {$a$};
	
	\draw[orientation_line] (0.25, -0.5) arc (-180:90:0.5);
	\draw (0.75, -0.5) node {$c$};
	
	\draw[orientation_line_opp] (3.9, 4.2) arc (-180:90:0.4);
	\draw (4.3, 4.2) node {$b$};
	
	\draw (1.75, 6) node {$B_1$};
	\draw (0.5, 5.5) node {$B_0$};
	\draw (-1.0, -1.5) node {$B_2$};
	\draw (1.5, -1.5) node {$B_3$};
	
	\draw (1, 4) node[left] {$\delta(a, b)$};
	\draw (6, 5) node[right] {$\delta(b, c)$};
	\draw (5, 1) node[right] {$\delta(a, bc)$};
	\draw (0, 0) node[left] {$\delta(ab, c)$};
\end{tikzpicture}
\end{center}
\caption{\label{Figure:VertexSimple}Coloured neighbourhood of the true vertex}
\end{figure}

By definition $[(v, \xi)] = \omega(a, b, c) = \delta(b, c) - \delta(ab, c) + \delta(a, bc) - \delta(a, b)$. So, to compute the value $[(v, \xi)]$, we should compute $\delta(x, y)$, where $x$ and $y$ are colours of cells in the neighbourhood of each edge near the vertex $v$ (with local orientations of cells matching local orientations of edges), and get this value with ''$+$`` if the local orientation of the edge is incoming to $v$, and with ''$-$`` otherwise. The correctness follows from the fact that $\delta(x, y) = -\delta(x^{-1}, xy) = \delta(y, (xy)^{-1}) = -\delta(y^{-1}, x^{-1}) = \delta((xy)^{-1}, x)$ for all $x, y\in G$. Finally, notice that in the sum $\sum\limits_{v\in\mathcal{V}(P)}[(v, \xi)]$ each value $\delta(x, y)$ comes with ''$+$`` for one vertex and with ''$-$`` for another vertex. So the total weight $w_{\xi}(P)$ is always zero.
\end{proof}

It follows from the proposition \ref{Proposition:DWTrivial} that the invariant $DW_{\omega}$ is defined by the cohomology class $[\omega]\in H^3(G; \ksmall)$. It does not change if we choose another representative of the class.

\begin{example}
\label{Example:LenseSpace}
Let $G = \mathbb{Z}_4 = \{0, 1, 2, 3\}$, and let $\K = \mathbb{Z}_2 = \{1, t\}$ ($t^2 = 1$). Consider the 3-cocycle $\omega\colon G^3\to \K$ such that $\omega(1, 1, 1) = t$ (hence $\omega(1, 2, 3) = \omega(2, 1, 2) = \omega(2, 3, 2) = \omega(3, 2, 1) = \omega(3, 3, 3) = t$) and for other $a, b, c\in G$: $\omega(a, b, c) = 1$. It's easy to check that $\omega$ is a 3-cocycle.

Let $M$ be a lens space $L_{4, 1}$, and let $P$ be a special spine of $M$, shown in the figure \ref{Figure:SpineL41} (the spine is represented as a neighbourhood of its singular graph).

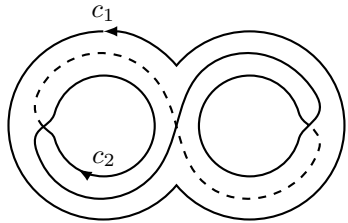
\begin{figure}[h]
\begin{center}
\begin{tikzpicture}[scale=0.75]
\begin{scope}[line, decoration={
    markings,
    mark=at position 0.95 with {\arrow{latex}}}
    ]
    
    \draw[line, postaction={decorate}] (5.9846,27.3001) .. controls
    (5.9186, 27.3464) and (5.8454, 27.4032) .. (5.8115, 27.5172) .. controls
    (5.7079, 27.8947) and (5.3623, 28.1720) .. (4.9519, 28.1720) .. controls
    (4.4596, 28.1720) and (4.0606, 27.7730) .. (4.0606, 27.2807) .. controls
    (4.0606, 26.7884) and (4.4596, 26.3894) .. (4.9519, 26.3894) .. controls
    (5.3612, 26.3894) and (5.7061, 26.6653) .. (5.8107, 27.0413) .. controls
    (5.8681, 27.3190) and (6.3875, 27.4079) .. (6.0926, 27.8713) .. controls
    (5.8831, 28.2108) and (5.5389, 28.4731) .. (5.1470, 28.5478) .. controls
    (4.9510, 28.5852) and (4.7457, 28.5757) .. (4.5555, 28.5157) .. controls
    (4.3653, 28.4556) and (4.1908, 28.3445) .. (4.0596, 28.1943) .. controls
    (3.8412, 27.9442) and (3.7555, 27.6086) .. (3.6670, 27.2886) .. controls
    (3.5941, 27.0255) and (3.5126, 26.7591) .. (3.3566, 26.5350) .. controls
    (3.2393, 26.3663) and (3.0812, 26.2259) .. (2.8989, 26.1312) .. controls
    (2.7166, 26.0365) and (2.5103, 25.9878) .. (2.3049, 25.9935) .. controls
    (2.0995, 25.9991) and (1.8956, 26.0592) .. (1.7212, 26.1678) .. controls
    (1.5468, 26.2764) and (1.4025, 26.4333) .. (1.3110, 26.6173) .. controls
    (0.9532, 27.2982) and (1.4541, 27.1889) .. (1.5352, 27.5201) .. controls
    (1.6398, 27.8961) and (1.9846, 28.1720) .. (2.3940, 28.1720) .. controls
    (2.8862, 28.1720) and (3.2853, 27.7730) .. (3.2853, 27.2807) .. controls
    (3.2853, 26.7884) and (2.8862, 26.3894) .. (2.3940, 26.3894) node[above] {$c_2$} .. controls
    (1.9836, 26.3894) and (1.6380, 26.6667) .. (1.5343, 27.0442) .. controls
    (1.5004, 27.1582) and (1.4429, 27.1584) .. (1.3417, 27.2585);    
\end{scope}

	\draw[line, dashed] (1.3309,27.2684) .. controls (1.1219, 27.4996) and (1.1606,
    27.7300) .. (1.2857, 27.9520) .. controls (1.4107, 28.1740) and (1.5468,
    28.2866) .. (1.7212, 28.3952) .. controls (1.8956, 28.5038) and (2.0995,
    28.5639) .. (2.3049, 28.5695) .. controls (2.5103, 28.5751) and (2.7166,
    28.5265) .. (2.8989, 28.4318) .. controls (3.0812, 28.3371) and (3.2393,
    28.1967) .. (3.3566, 28.0280) .. controls (3.5126, 27.8039) and (3.5941,
    27.5375) .. (3.6670, 27.2744) .. controls (3.7555, 26.9543) and (3.8412,
    26.6188) .. (4.0596, 26.3687) .. controls (4.1908, 26.2185) and (4.3653,
    26.1074) .. (4.5555, 26.0473) .. controls (4.7457, 25.9873) and (4.9510,
    25.9778) .. (5.1470, 26.0152) .. controls (5.5389, 26.0898) and (5.8972,
    26.3296) .. (6.0768, 26.6664) .. controls (6.2563, 27.0031) and (6.1887,
    27.0896) .. (6.0082, 27.2891);

  \draw[line, -latex] (2.3845,28.9534) node[above] {$c_1$} .. controls
    (1.4659, 28.9534) and (0.7212, 28.2088) .. (0.7212, 27.2902) .. controls
    (0.7213, 26.3716) and (1.4659, 25.6270) .. (2.3845, 25.6270) .. controls
    (2.8809, 25.6273) and (3.3512, 25.8493) .. (3.6670, 26.2324) .. controls
    (3.9827, 25.8493) and (4.4530, 25.6273) .. (4.9494, 25.6270) .. controls
    (5.8680, 25.6269) and (6.6127, 26.3716) .. (6.6127, 27.2902) .. controls
    (6.6127, 28.2088) and (5.8680, 28.9535) .. (4.9494, 28.9534) .. controls
    (4.4532, 28.9531) and (3.9831, 28.7313) .. (3.6674, 28.3485) .. controls
    (3.3515, 28.7315) and (2.8810, 28.9534) .. (2.3845, 28.9534);

\end{tikzpicture}
\end{center}
\caption{\label{Figure:SpineL41}Special spine of the lens space $L_{4, 1}$}
\end{figure}

The spine $P$ contains two 2-components, denoted by $c_1$ and $c_2$, and only one true vertex $v$. Let $c_1$ be a 2-component of length 2, and $c_2$ be a 2-component of length 4. There are only four colourings of $P$ by elements of the group $G$ (in this case orientations of 2-components and triple lines do not matter): $$(c_1, c_2) \mapsto (0, 0), (0, 2), (2, 1), (2, 3).$$

For the first two colourings the weights of the polyhedron $P$ are $\omega(0, 0, 0) = 1$ and $\omega(2, 0, 2) = 1$. For the last two colours the weights are $\omega(1, 2, 3) = t$ and $\omega(3, 2, 1) = t$. So $DW_{\omega}(M) = 2 + 2t$.
\end{example}

\end{document}